\documentclass{article}
\usepackage{cite}
\usepackage{amsmath,amssymb,amsfonts,amsthm}
\usepackage{algorithmic}
\usepackage{graphicx}
\usepackage{algorithm,algorithmic}
\usepackage{hyperref,pdfsync}
\usepackage{textcomp}
\usepackage{dsfont}

\usepackage{epstopdf}
\usepackage{subfigure}
\graphicspath{{./}{./figures/}}

\allowdisplaybreaks

\newtheorem{nnassumption}{\bf Assumption}

\newtheorem{nntheorem}{\bf Theorem}
\newenvironment{theorem}{\begin{nntheorem}\it}{\end{nntheorem}}
\newtheorem{nncorollary}{\bf Corollary}

\newtheorem{nndefinition}{\bf Definition}

\newtheorem{nnproposition}{\bf Proposition}

\newtheorem{nnproblem}{\bf Problem}

\newtheorem{nnlemma}{\bf Lemma}
\newenvironment{lemma}{\begin{nnlemma}\it}{\end{nnlemma}}
\newtheorem{nnremark}{\bf Remark}
\newenvironment{remark}{\begin{nnremark} \rm }{\hfill \hspace*{1pt}\hfill $\circ$\end{nnremark}}
\newtheorem{nnexample}{\bf Example}

\renewcommand{\leq}{\leqslant}
\renewcommand{\geq}{\geqslant}

\begin{document}
\title{Stabilization of a Wave-Heat Cascade System}
\author{Hugo Lhachemi, Christophe Prieur, and Emmanuel Tr{\'e}lat
\thanks{Hugo Lhachemi is with Universit{\'e} Paris-Saclay, CNRS, CentraleSup{\'e}lec, Laboratoire des signaux et syst\`emes, 91190, Gif-sur-Yvette, France (email: hugo.lhachemi@centralesupelec.fr).}
\thanks{Christophe Prieur is with Universit\'e Grenoble Alpes, CNRS, Gipsa-lab, 38000 Grenoble, France (e-mail: christophe.prieur@gipsa-lab.fr).}
\thanks{Emmanuel Trélat is with Sorbonne Universit\'e, Universit\'e Paris Cit\'e, CNRS, Inria, Laboratoire Jacques-Louis Lions, LJLL, F-75005 Paris, France (e-mail: emmanuel.trelat@sorbonne-universite.fr).}}

\date{}

\maketitle

\begin{abstract}
We consider the output‑feedback stabilization of a one‑dimensional cascade coupling a reaction-diffusion equation and a wave equation through an internal term, with Neumann boundary control acting at the wave endpoint. Two measurements are available: the wave velocity at the controlled boundary and a temperature‑type observation of the reaction-diffusion component, either distributed or pointwise. Under explicit, necessary and sufficient conditions on the coupling and observation profiles, we show that the generator of the open‑loop system is a Riesz‑spectral operator. Exploiting this structure, we design a finite‑dimensional dynamic output‑feedback law, based on a finite number of parabolic modes, which achieves arbitrary exponential decay in both the natural energy space and a stronger parabolic norm. The construction relies on a spectral reduction and a Lyapunov argument in Riesz bases. We also extend the design to pointwise temperature or heat-flux measurements. 
\end{abstract}


~\newline\textbf{The present work is the result of the split of \cite{lhachemi2025controllability} into the feedback stabilization study reported in this paper and the study of the controllability properties of the cascade in the companion work \cite{LPT_ongoing}.}

\section{Introduction}\label{sec:introduction}

\subsection{Studied problem}

We fix $L>0$, $c\in\mathbb{R}$ and $\beta\in L^\infty(0,L)$. We consider the cascade system
\begin{subequations}\label{eq: cascade equation}
    \begin{align}
        \partial_t y(t,x) &= \partial_{xx} y(t,x) + c y(t,x) + \beta(x) z(t,x) , \label{eq: cascade equation - 1} \\
        \partial_{tt} z(t,x) &= \partial_{xx} z(t,x) , \label{eq: cascade equation - 2}
    \end{align}
for $t > 0$ and $x \in (0,L)$, with boundary conditions
    \begin{align}
        & y(t,0) = y(t,L) = 0 , \label{eq: cascade equation - 3} \\ 
        & z(t,0) = 0 , \ \ \partial_x z(t,L) = u(t) . \label{eq: cascade equation - 4} 
    \end{align}
\end{subequations}
The state $y(t,\cdot)$ solves a one‑dimensional reaction-diffusion equation, forced by the wave state $z(t,\cdot)$ that solves a one‑dimensional wave equation. The control input $u$ acts on the right Neumann trace of $z$. Such wave-reaction-diffusion cascades appear for instance in simplified models of microwave heating, see e.g.~\cite{hill1996modelling, wei2012optimal, zhong2014state, celuch2009modeling}.
We assume that two outputs are available: the first one is the wave velocity at the right endpoint,
\begin{equation}\label{eq: measurement wave}
    z_o(t) = \partial_t z(t,L) ,
\end{equation} 
and the second one is a heat measurement of the form
\begin{equation}\label{eq: measurement heat distributed}
    y_o(t) = \int_0^L c_o(x) y(t,x) \,\mathrm{d}x 
\end{equation}
where $c_o \in L^2(0,L)$ models a weighted temperature observation.

\subsection{State-of-the-art and contribution}
To the best of our knowledge, controllability and stabilization results for coupled PDEs of different types (such as heat-wave cascades) remain rather scarce. This contrasts with the rich theory available for coupled parabolic systems; see for instance \cite{boyer,bhandari2021boundary,lhachemi2025boundary} and the references therein.  The works \cite{zhang2003polynomial, zhang2004polynomial} on a heat-wave system coupled at the boundary have strongly inspired the present paper. Hyperbolic-elliptic couplings have been investigated in \cite{rosier2013unique, chowdhury2023boundary}. Backstepping‑based stabilization has been developed in \cite{chen2017backstepping, ghousein2020backstepping} for hyperbolic-parabolic systems where the hyperbolic component is a transport equation.

To the best of the authors’ knowledge, the present paper is the first to address the design of an explicit output‑feedback controller for a one‑dimensional wave-heat cascade. Our main contribution is the construction of a finite‑dimensional dynamic output‑feedback law ensuring exponential stabilization of the full cascade \eqref{eq: cascade equation}, with an arbitrary prescribed decay rate, under spectral controllability and observability conditions that are explicit, necessary, and sufficient. The key structural ingredient is that the generator of the closed‑loop system is a Riesz‑spectral operator. We establish the following result that, at this stage, we state in a rather informal way.

\begin{theorem}\label{thm: informal}
Given any $\delta>0$, under explicit controllability and observability conditions, there exists an output-feedback control, using the measurements \eqref{eq: measurement wave} and \eqref{eq: measurement heat distributed}, explicitly built from a finite number of modes (depending on $\delta$), such that the closed-loop controlled system \eqref{eq: cascade equation} is exponentially stable in $\mathcal{H}^0 = L^2(0,L) \times H_{(0)}^1(0,L) \times L^2(0,L)$ and also in $\mathcal{H}^1 = H_0^1(0,L) \times H_{(0)}^1(0,L) \times L^2(0,L)$, with the decay rate $\delta$. 
\end{theorem}

The precise statement is given in Theorem~\ref{thm: main result 2} in Section~\ref{sec: control design}. Based on spectral considerations, we design a Lyapunov function and build the feedback control using only a finite number of modes. We also extend the result of Theorem~\ref{thm: informal} to the case of a pointwise temperature measurement $y_o(t) = y(t,\xi_p)$, or a pointwise heat flux measurement $y_o(t) = \partial_x y(t,\xi_p)$, at some location $\xi_p \in [0,L]$.

\subsection{Paper organization and technical outline}

The paper is organized as follow. In Section~\ref{sec: prel control and spectral prop}, we design a preliminary boundary velocity feedback acting on the wave equation and analyze the spectrum of the resulting generator. Section~\ref{sec: spectral reduction} is devoted to the spectral reduction of the problem and to the derivation of finite‑dimensional models capturing the unstable parabolic modes. The output‑feedback design and the stability analysis of the closed‑loop system are carried out in Section~\ref{sec: control design}. In Section~\ref{sec: extemsion} we extend the results from distributed temperature measurements to pointwise temperature or heat‑flux measurements. Section~\ref{sec: numerical} presents a numerical example illustrating the performance of the proposed controller. Section~\ref{sec: Conclusion} gathers some concluding remarks and perspectives. In the Appendix, we provide further comments on the exact controllability properties of the wave-heat cascade.

For the reader's convenience, we provide below a brief technical outline of the main steps leading to our stabilization result (Theorem~\ref{thm: main result 2} in Section~\ref{sec: control design}):
\begin{itemize}
\item Using the boundary wave measurement \eqref{eq: measurement wave}, we first design a preliminary velocity feedback \eqref{eq: preliminary feedback} to stabilize the hyperbolic part and define the corresponding generator $\mathcal{A}$.
\item We compute the eigenvalues and eigenvectors of $\mathcal{A}$ (Lemma~\ref{lem: eigenstructures of A}).
\item We prove that the family of eigenvectors of $\mathcal{A}$ forms a Riesz basis $\Phi$ of the energy space $\mathcal{H}^0$, so that $\mathcal{A}$ is a Riesz‑spectral operator generating a $C_0$-semigroup (Lemma~\ref{lem: A riesz spectral}).
\item We compute the dual Riesz basis $\Psi$ by analysing the adjoint operator $\mathcal{A}^*$ (Lemma~\ref{lem: dual basis}).
\item For the parabolic component, we construct a Riesz basis $\Phi^1$ of a stronger Hilbert space $\mathcal{H}^1$ (Lemma~\ref{lemma: Riesz basis H1}).
\item We perform a spectral reduction of the closed‑loop system and identify finite‑dimensional models capturing the relevant parabolic modes (Subsections~\ref{sub sec: Finite-dimensional reduced model} and~\ref{subsec: finite-dim model}).
\item We derive necessary and sufficient controllability and observability conditions for these models (Lemmas~\ref{kalman:contro:lemma}-\ref{lem: obs for cont}).
\item We introduce a finite‑dimensional output‑feedback controller based on a Luenberger‑type observer and a state‑feedback on the truncated model (Section~\ref{sec: control design}).
\item Finally, we prove exponential stability of the full PDE cascade in both $\mathcal{H}^0$ and $\mathcal{H}^1$ norms by means of a Lyapunov functional built in the Riesz bases (Theorem~\ref{thm: main result 2}). 
\end{itemize}

\section{Preliminary velocity feedback and spectral properties}\label{sec: prel control and spectral prop}

\subsection{Preliminary velocity feedback}

The control strategy consists of first shifting the spectrum of the wave equation \eqref{eq: cascade equation - 2}, by applying a velocity feedback on the right Neumann trace \eqref{eq: cascade equation - 4}, namely, by setting
\begin{equation}\label{eq: preliminary feedback}
    u(t) 
    = -\alpha z_o(t) + v(t)
    = -\alpha \partial_t z(t,L) + v(t)
\end{equation}
where $v$ is another control to be designed, and $\alpha > 1$ is chosen such that (see Remark~\ref{rem: simple eigenvalues})
\begin{equation}\label{eq: condition for selecting alpha}
    c \neq \tfrac{1}{2L} \log\big(\tfrac{\alpha-1}{\alpha+1}\big) + \tfrac{n^2\pi^2}{L^2}  \qquad \forall n\in\mathbb{N}^{*} .
\end{equation}
When $v=0$, the output-feedback $u$ given by \eqref{eq: preliminary feedback} is known to stabilize the wave equation \eqref{eq: cascade equation - 2}. The additional control $v$ is going to be used and designed from the output \eqref{eq: measurement heat distributed} in order to stabilize the full cascade system \eqref{eq: cascade equation}. Hence, with \eqref{eq: preliminary feedback}, the system is now
\begin{subequations}\label{eq: cascade equation - premilinary feedback}
    \begin{align}
        &\partial_t y(t,x) = \partial_{xx} y(t,x) + c y(t,x) + \beta(x) z(t,x) , \label{eq: cascade equation - premilinary feedback - 1} \\
        &\partial_{tt} z(t,x) = \partial_{xx} z(t,x) , \label{eq: cascade equation - premilinary feedback - 2} \\
        & y(t,0) = y(t,L) = 0 , \label{eq: cascade equation - premilinary feedback - 3} \\ 
        & z(t,0) = 0 , \quad \partial_x z(t,L) + \alpha \partial_t z(t,L) = v(t) , \label{eq: cascade equation - premilinary feedback - 4} 
    \end{align}
\end{subequations}
and the new control is $v$. 

To design an effective feedback control $v$ based on a finite number of modes, we follow a spectral reduction approach developed in \cite{coron2004global, coron2006global, lhachemi2020pi, russell1978controllability} in other contexts. In turn, we will design a Lyapunov function.

We define the Hilbert space of complex-valued functions
\begin{subequations}\label{eq: state-space}
    \begin{equation}
        \mathcal{H}^0 = L^2(0,L) \times H_{(0)}^1(0,L) \times L^2(0,L)
    \end{equation}
    where $H_{(0)}^1(0,L) = \{ g \in H^1(0,L)\, \mid\, g(0) = 0 \}$, with the inner product
    \begin{equation}
	\big\langle (f_1,g_1,h_1),(f_2,g_2,h_2)\big\rangle_{\mathcal{H}^0}
        = \int_0^L (f_1 \overline{f_2} + g'_1 \overline{g'_2} + h_1 \overline{h_2} )
    \end{equation}
\end{subequations}
and corresponding norm denoted $\Vert \cdot \Vert_{\mathcal{H}^0}$. When the context is clear, we simply denote by $\langle \cdot , \cdot \rangle$ the inner product of $\mathcal{H}_0$. Setting $X=(y,z,\partial_t z)^\top$, the control system \eqref{eq: cascade equation - premilinary feedback} is written in the abstract form $\dot{X}(t)=\mathcal{A}X(t)+\mathcal{B}v(t)$, where the operator $\mathcal{A}:D(\mathcal{A}) \subset\mathcal{H}^0 \rightarrow\mathcal{H}^0$ is defined by
\begin{subequations}\label{def_A}
\begin{equation}
    \mathcal{A} = \begin{pmatrix} \partial_{xx} + c\,\mathrm{id} & \beta\,\mathrm{id} & 0 \\ 0 & 0 & \mathrm{id} \\ 0 & \partial_{xx} & 0 \end{pmatrix}
\end{equation}
($\mathrm{id}$ denotes the identity operator on $L^2(0,L)$)
and domain
\begin{multline}
  \!\!\!  D(\mathcal{A}) = \big\{ (f,g,h)\in H^2(0,L) \times H^2(0,L) \times H^1(0,L)\, \mid\,  \\
    f(0) \!=\! f(L) \!=\! g(0) \!=\! h(0) \!=\! 0 , \   g'(L) \!+\! \alpha h(L) \!=\! 0 \big\} ,
\end{multline}
\end{subequations}
and the control operator $\mathcal{B}$ 
is defined by transposition (see Appendix~\ref{sec_app_A}).

\subsection{Spectral properties of $\mathcal{A}$}\label{sec_prelim_A}

\begin{lemma}\label{lem: eigenstructures of A}
The eigenvalues of $\mathcal{A}$ are 
\begin{align*}
& \lambda_{1,n} = c - \tfrac{n^2 \pi^2}{L^2}, \quad n\in\mathbb{N}^*, \\ 
& \lambda_{2,m} = \tfrac{1}{2L} \log\big(\tfrac{\alpha-1}{\alpha+1}\big) + i \tfrac{m\pi}{L} , \quad m \in\mathbb{Z},
\end{align*}
with associated eigenvectors $\phi_{1,n}=(\phi^1_{1,n},\phi^2_{1,n},\phi^3_{1,n})$ and $\phi_{2,m}=(\phi^1_{2,m},\phi^2_{2,m},\phi^3_{2,m})$ respectively given by
\begin{align*}
\phi^1_{1,n}(x) = & \sqrt{\tfrac{2}{L}} \sin\big(\tfrac{n\pi}{L}x\big), \quad \phi^2_{1,n}(x) = 0, \quad \phi^3_{1,n}(x) = 0 , \\
\phi^1_{2,m}(x) = & \tfrac{1}{A_m r_m} \int_x^L \beta(s) \sinh(\lambda_{2,m} s) \sinh(r_m (x-s)) \,\mathrm{d}s \\
& \; + \tfrac{\sinh(r_m (L-x))}{A_m r_m \sinh(r_m L)} \int_0^L \beta(s) \sinh(\lambda_{2,m} s) \sinh(r_m s) \,\mathrm{d}s 
\end{align*}
where $r_m$ is a square root of $\lambda_{2,m} - c$ with $\operatorname{Re}(r_m)\geq 0$,
\begin{align*}
&\phi^2_{2,m}(x) = \tfrac{1}{A_m}\sinh(\lambda_{2,m} x) , \quad \phi^3_{2,m}(x) = \tfrac{\lambda_{2,m}}{A_m} \sinh(\lambda_{2,m} x)  , \\
&A_m = \tfrac{\sqrt{(\mu^2 L^2 + m^2 \pi^2)\sinh(2\mu L)}}{L\sqrt{2\mu}} , \quad
\mu = - \tfrac{1}{2L} \log\big(\tfrac{\alpha-1}{\alpha+1}\big) > 0 .
\end{align*}
\end{lemma}

\begin{remark}\label{rem: simple eigenvalues}
The constraint \eqref{eq: condition for selecting alpha} is introduced to avoid that the real eigenvalue $\lambda_{2,0}$ associated with the wave equation coincide with an eigenvalue $\lambda_{1,n}$ of the reaction-diffusion equation. When equality holds for some $n\in\mathbb{N}^*$, the eigenvalue $\lambda_{2,0} = \lambda_{1,n}$ is of geometric multiplicity one but of algebraic multiplicity two, leading to a Jordan block of dimension 2. Since all other eigenvalues are simple, our approach works. It also applies to the case $\alpha \in (0,1)$ with the change $\lambda_{2,m} = \tfrac{1}{2L} \log\big(\tfrac{1-\alpha}{1+\alpha}\big) + i \tfrac{(2m+1)\pi}{2L}$ for $m \in\mathbb{Z}$.
\end{remark}

\begin{proof}
Let $\lambda\in\mathbb{C}$ and $(f,g,h)\in D(\mathcal{A})$ be such that $\mathcal{A}(f,g,h)=\lambda (f,g,h)$, i.e., $f,g \in H^2(0,L)$ and $h\in H^1(0,L)$ such that
\begin{subequations}\label{eq: lemma eigenstructures A - system to solve}
    \begin{align}
        & f''+cf+\beta g = \lambda f , \label{eq: lemma eigenstructures A - system to solve - 1} \\
        & h = \lambda g , \label{eq: lemma eigenstructures A - system to solve - 2} \\
        & g'' = \lambda h = \lambda^2 g , \label{eq: lemma eigenstructures A - system to solve - 3} \\
        & f(0) = f(L) = g(0) = h(0) = 0 , \label{eq: lemma eigenstructures A - system to solve - 4} \\
        & g'(L) + \alpha h(L) = g'(L) + \alpha \lambda g(L) = 0 . \label{eq: lemma eigenstructures A - system to solve - 5}
    \end{align}
\end{subequations}
Assume first that $\lambda = 0$. By \eqref{eq: lemma eigenstructures A - system to solve - 3} we have $g''=0$ with $g(0) = g'(L) = 0$, implying along with \eqref{eq: lemma eigenstructures A - system to solve - 2} that $g=h=0$. Hence, from \eqref{eq: lemma eigenstructures A - system to solve - 1}, $f''+cf=0$ and $f(0)=f(L)=0$, from which we deduce that $\tfrac{n^2 \pi^2}{L^2} = c$ and $f(x) = \sqrt{\tfrac{2}{L}}\sin\big(\tfrac{n\pi}{L}x\big)$ for some $n\in\mathbb{N}^*$. In particular, $\lambda = c - \tfrac{n^2\pi^2}{L^2} = 0$.

Assume now that $\lambda \neq 0$. Using \eqref{eq: lemma eigenstructures A - system to solve - 3} and the fact that $g(0)=0$, we have $g(x) = \delta ( e^{\lambda x} - e^{-\lambda x} )$ for some $\delta\in\mathbb{R}$. Furthermore, by \eqref{eq: lemma eigenstructures A - system to solve - 5},
$0 = g'(L) + \alpha \lambda g(L) = \lambda \delta ( e^{\lambda L} + e^{-\lambda L} + \alpha ( e^{\lambda L} - e^{-\lambda L} ) )$.
Since $\lambda \neq 0$, there are two cases. 

If $\delta = 0$ then $g=h=0$ and, by \eqref{eq: lemma eigenstructures A - system to solve - 1}, $f''=(\lambda-c)f$ and $f(0)=f(L)=0$ hence $\lambda = c - \tfrac{n^2\pi^2}{L^2}$ and $f(x) = \sqrt{\tfrac{2}{L}}\sin\big(\tfrac{n\pi}{L}x\big)$ for some $n\in\mathbb{N}^*$ provided $\lambda \neq 0$. 

If $\delta \neq 0$ then $e^{\lambda L} + e^{-\lambda L} + \alpha ( e^{\lambda L} - e^{-\lambda L} ) = 0$, which is equivalent to $e^{2\lambda L} = \tfrac{\alpha-1}{\alpha+1}$. Since $\alpha > 1$, this gives $\lambda = \tfrac{1}{2L} \log\big(\tfrac{\alpha-1}{\alpha+1}\big) + i \tfrac{m\pi}{L}$ for some $m \in\mathbb{Z}$. Moreover $g(x) = \sinh(\lambda x)$ and $h(x) = \lambda \sinh(\lambda x)$, where we take (without loss of generality) $\delta=1/2$. Finally, $f''+(c-\lambda)f=-\beta(x)g$ with $f(0)=f(L)=0$. Recalling that $\alpha > 1$ is selected so that $c \neq \tfrac{1}{2L} \log\big(\tfrac{\alpha-1}{\alpha+1}\big)$, we have $\lambda-c \neq 0$. Hence, denoting by $r$ one of its two distinct square roots, i.e., $r^2 = \lambda - c$ with $r \neq 0$, we obtain
$f(x) = \big( \delta_1 - \tfrac{1}{2r} \int_L^x \beta(s) g(s) e^{-r s} \,\mathrm{d}s \big) e^{r x} + \big( \delta_2 + \tfrac{1}{2r} \int_L^x \beta(s) g(s) e^{r s} \,\mathrm{d}s \big) e^{-r x}$
for some constants $\delta_1,\delta_2\in\mathbb{C}$ that must be selected such that $f(0)=0$ and $f(L)=0$. The latter equation yields $\delta_2 = - \delta_1 e^{2rL}$ , implying that
$f(x) = 2 \delta_1 e^{rL} \sinh(r(x-L)) - \tfrac{1}{r} \int_L^x \beta(s) g(s) \sinh(r(x-s)) \,\mathrm{d}s$.
Then, $f(0)=0$ gives $-2 \delta_1 e^{rL} \sinh(rL) - \tfrac{1}{r} \int_0^L \beta(s) g(s) \sinh(rs) \,\mathrm{d}s = 0$. We note that $\sinh(rL)=0$ if and only if $e^{2rL}=1$, i.e., if and only if $2rL=2ik\pi$ for some $k\in\mathbb{Z}$. Then, we must have $\lambda = c + r^2 = c - \tfrac{k^2 \pi^2}{L^2}$. Since $\lambda = \tfrac{1}{2L} \log\big(\tfrac{\alpha-1}{\alpha+1}\big) + i \tfrac{m\pi}{L}$, this is possible only if $m=0$, which contradicts the assumption that $\alpha > 1$ has been selected so that $c \neq \tfrac{1}{2L} \log\big(\tfrac{\alpha-1}{\alpha+1}\big) + \tfrac{k^2\pi^2}{L^2}$ for any $k\in\mathbb{N}$. Hence, $\sinh(rL) \neq 0$, which gives $\delta_1 = - \tfrac{1}{2 e^{rL} r \sinh(r L)} \int_0^L \beta(s) g(s) \sinh(rs) \,\mathrm{d}s$. 
\end{proof}

\begin{remark}\label{rem_spectrum}
The sequence of eigenvalues $(\lambda_{1,n})_{n\in\mathbb{N}^*}$ (resp., $(\lambda_{1,m})_{m\in\mathbb{Z}}$) is associated with the reaction-diffusion part \eqref{eq: cascade equation - 1} (resp., the wave part \eqref{eq: cascade equation - 2}) of the system \eqref{eq: cascade equation} and is referred to as the \emph{parabolic spectrum} (resp., the \emph{hyperbolic spectrum}). 
\end{remark}

\begin{lemma}\label{lem: A riesz spectral}
The family of generalized eigenvectors $\Phi = \{\phi_{1,n}\, \mid\, n\in\mathbb{N}^* \} \cup \{\phi_{2,m}\, \mid\, m \in\mathbb{Z} \}$ is a Riesz basis of $\mathcal{H}^0$. Hence, $\mathcal{A}$ is a Riesz spectral operator that generates a $C_0$-semigroup.
\end{lemma}


\begin{proof}
Noting that $\{\phi^1_{1,n}\, \mid\, n\in\mathbb{N}^* \}$ is a Hilbert basis of $L^2(0,L)$ and that $\{ (\phi^2_{2,m},\phi^3_{2,m})\, \mid\, m\in\mathbb{Z} \}$ is a Riesz basis of $H_{(0)}^1(0,L) \times L^2(0,L)$ we infer, first, that $\Phi$ is $\omega$-linearly independent, and second, defining $\tilde{\phi}_{1,n} = \phi_{1,n}$ and $\tilde{\phi}_{2,m} = (0,\phi^2_{2,m},\phi^3_{2,m})$, that $\tilde{\Phi} = \{\tilde{\phi}_{1,n}\, \mid\, n\in\mathbb{N}^* \} \cup \{\tilde{\phi}_{2,m}\, \mid\, m \in\mathbb{Z} \}$ is a Riesz basis of $\mathcal{H}^0$. By Bari's theorem \cite{gohberg1978introduction}, it then follows that $\Phi$ is a Riesz basis provided that
\begin{equation}\label{bari_ineq}
\sum_{n\in\mathbb{N}^*} \Vert \phi_{1,n} - \tilde{\phi}_{1,n} \Vert_{\mathcal{H}^0}^2 + \sum_{m\in\mathbb{Z}} \Vert \phi_{2,m} - \tilde{\phi}_{2,m} \Vert_{\mathcal{H}^0}^2 
= \sum_{m\in\mathbb{Z}} \Vert \phi^1_{2,m} \Vert_{L^2}^2
< \infty .
\end{equation}
To prove \eqref{bari_ineq}, we first note that $\phi^1_{2,m} = f_{m,1} + f_{m,2}$ with 
\begin{align*}
& f_{m,1}(x) \!=\! \tfrac{\sinh(r_m (L-x))}{A_m r_m \sinh(r_m L)} \! \int_0^x \!\! \beta(s) \sinh(\lambda_{2,m} s) \sinh(r_m s) \,\mathrm{d}s , \\
& f_{m,2}(x) = \tfrac{1}{A_m r_m \sinh(r_m L)} \\
& \hspace{2cm} \times \int_x^L \beta(s) \sinh(\lambda_{2,m} s) \big( \sinh(r_m (L-x) ) \sinh(r_m s) - \sinh(r_m L) \sinh(r_m (s-x)) \big) \,\mathrm{d}s .
\end{align*}
We study the two terms separately. Since $r_m^2 = \lambda_{2,m} - c = \tfrac{1}{2L} \log\big(\tfrac{\alpha-1}{\alpha+1}\big) - c  + i \tfrac{m\pi}{L}$ with $\operatorname{Re}(r_m)\geq 0$, we infer that $\vert r_m \vert \sim \sqrt{\tfrac{\vert m \vert \pi}{L}}$ and $\operatorname{Re} r_m \sim \sqrt{\tfrac{\vert m \vert \pi}{2L}}$ as $\vert m \vert \rightarrow +\infty$. Recall that $\vert \sinh(\operatorname{Re} z) \vert \leq \vert \cosh(z) \vert , \vert \sinh(z) \vert \leq \vert \cosh(\operatorname{Re} z) \vert$ for any $z \in \mathbb{C}$, and $\cosh(x) \leq e^x$ and $\tfrac{1}{2}(e^x-1) \leq \sinh(x) \leq e^x /2$ for any $x \geq 0$. For $\vert m \vert$ large, we obtain
\begin{align*}
\vert f_{m,1}(x) \vert
& \leq \tfrac{\vert \sinh(r_m (L-x)) \vert}{A_m \vert r_m \vert \vert \sinh(r_m L) \vert}  \int_0^x \vert \beta(s) \vert \vert \sinh(\lambda_{2,m} s) \vert \vert \sinh(r_m s) \vert \,\mathrm{d}s \\
& \leq \tfrac{\Vert \beta \Vert_{L^\infty} \cosh(\operatorname{Re}r_m (L-x))}{A_m \vert r_m \vert \sinh(\operatorname{Re}r_m L)} \!\! \int_0^x \!\!\! \cosh(\operatorname{Re}\lambda_{2,m} s) \cosh(\operatorname{Re}r_m s) \,\mathrm{d}s \\
& \leq \tfrac{2\Vert \beta \Vert_{L^\infty} \cosh(\mu L) e^{\operatorname{Re}r_m (L-x)}}{A_m \vert r_m \vert (e^{\operatorname{Re}r_m L}-1)} \int_0^x \cosh(\operatorname{Re}r_m s) \,\mathrm{d}s \\
& \leq \tfrac{2\Vert \beta \Vert_{L^\infty} \cosh(\mu L) e^{\operatorname{Re}r_m (L-x)}}{A_m \vert r_m \vert (e^{\operatorname{Re}r_m L}-1)} \times \tfrac{\sinh(\operatorname{Re}r_m x)}{\operatorname{Re}r_m} \\
& \leq \tfrac{\Vert \beta \Vert_{L^\infty} \cosh(\mu L) e^{\operatorname{Re}r_m (L-x)}}{A_m \vert r_m \vert (e^{\operatorname{Re}r_m L}-1)} \ \tfrac{e^{\operatorname{Re}r_m x}}{\operatorname{Re}r_m} \\
& \leq \Vert \beta \Vert_{L^\infty} \cosh(\mu L) \ \tfrac{e^{\operatorname{Re}r_m L}}{e^{\operatorname{Re}r_m L}-1} \ \tfrac{1}{A_m \vert r_m \vert \operatorname{Re}r_m}
\end{align*}
hence $\Vert f_{m,1} \Vert_{L^\infty} = O(1/m^2)$ as $\vert m \vert \rightarrow +\infty$. For the second term, noting that $\sinh(r_m (L-x) ) \sinh(r_m s) - \sinh(r_m L) \sinh(r_m (s-x)) = \sinh(r_m x) \sinh(r_m(L-s))$, we have
\begin{align*}
\vert f_{m,2}(x) \vert 
& \leq \tfrac{\int_x^L \vert \beta(s) \vert \vert \sinh(\lambda_{2,m} s) \vert \vert \sinh(r_m x) \vert \vert \sinh(r_m(L-s)) \vert \,\mathrm{d}s }{A_m \vert r_m \vert \vert \sinh(r_m L) \vert} \\
& \leq \tfrac{\Vert \beta \Vert_{L^\infty} \int_x^L \cosh( \operatorname{Re}\lambda_{2,m} s) \cosh(\operatorname{Re}r_m x) \cosh(\operatorname{Re}r_m(L-s)) \,\mathrm{d}s }{A_m \vert r_m \vert \sinh(\operatorname{Re}r_m L)}  \\
& \leq \tfrac{2 \Vert \beta \Vert_{L^\infty} \cosh( \mu L) e^{\operatorname{Re}r_m x}}{A_m \vert r_m \vert ( e^{\operatorname{Re}r_m L}-1)} \int_x^L e^{\operatorname{Re}r_m(L-s)} \,\mathrm{d}s  \\
& \leq \tfrac{2 \Vert \beta \Vert_{L^\infty} \cosh( \mu L) e^{\operatorname{Re}r_m x}}{A_m \vert r_m \vert ( e^{\operatorname{Re}r_m L}-1)} \ \tfrac{e^{\operatorname{Re}r_m(L-x)}}{\operatorname{Re}r_m} \\  
& \leq 2 \Vert \beta \Vert_{L^\infty} \cosh( \mu L) \tfrac{e^{\operatorname{Re}r_m L}}{ e^{\operatorname{Re}r_m L}-1} \ \tfrac{1}{A_m \vert r_m \vert \operatorname{Re}r_m}
\end{align*}
hence $\Vert f_{m,2} \Vert_{L^\infty} = O(1/m^2)$ as $\vert m \vert \rightarrow +\infty$. 
\end{proof}


Since the controllability properties of the system are captured by the properties of the dual Riesz basis $\Psi$ of $\Phi$, we next compute the eigenvectors of the adjoint operator $\mathcal{A}^*$.

\begin{lemma}\label{lem: dual basis}
Identifying the Hilbert space $\mathcal{H}^0$ with its dual, the adjoint operator $\mathcal{A}^*$ is given by
\begin{subequations}\label{def_A*}
    \begin{equation}
        \mathcal{A}^* = \begin{pmatrix} \partial_{xx} + c\,\mathrm{id} & 0 & 0 \\ P_\beta & 0 & - \mathrm{id} \\ 0 & -\partial_{xx} & 0 \end{pmatrix}
    \end{equation}
    with $P_\beta f = \int_0^{(\cdot)} \int_\tau^L \beta(s) f(s) \,\mathrm{d}s\,\mathrm{d}\tau$, and domain
\begin{multline}
\!\!\! D(\mathcal{A^*}) = \{ (f,g,h)\in H^2(0,L) \times H^2(0,L) \times H^1(0,L) \,\mid\,  \\
 f(0) \!=\! f(L) \!=\! g(0) \!=\! h(0) \!=\! 0 , \ g'(L) \!-\! \alpha h(L) \!=\! 0 \} .
\end{multline}
\end{subequations}
Its eigenfunctions are given by the dual Riesz basis $\Psi = \{\psi_{1,n}\, \mid\, n\in\mathbb{N}^* \} \cup \{\psi_{2,m}\, \mid\, m \in\mathbb{Z} \}$ of $\Phi$, associated with the eigenvalues 
\begin{align*}
& \mu_{1,n} = \lambda_{1,n} = c - \tfrac{n^2 \pi^2}{L^2}, \qquad n\in\mathbb{N}^* , \\ 
& \mu_{2,m} = \bar\lambda_{2,m} = \tfrac{1}{2L} \log\big(\tfrac{\alpha-1}{\alpha+1}\big) - i \tfrac{m\pi}{L} , \qquad m \in\mathbb{Z} ,
\end{align*}
where, setting $\psi_{1,n}=(\psi^1_{1,n},\psi^2_{1,n},\psi^3_{1,n})$ and $\psi_{2,m}=(\psi^1_{2,m},\psi^2_{2,m},\psi^3_{2,m})$,
\begin{align*}
\psi^1_{1,n}(x) = & \sqrt{\tfrac{2}{L}} \sin\big(\tfrac{n\pi}{L}x\big) , \\ 
\psi^2_{1,n}(x) = & - \tfrac{\gamma_n}{\lambda_{1,n}^2 \big( \cosh(\lambda_{1,n} L) + \alpha \sinh(\lambda_{1,n} L) \big) } \sqrt{\tfrac{2}{L}} \big( \cosh(\lambda_{1,n} (x-L)) - \alpha \sinh(\lambda_{1,n} (x-L)) \big) \\
& - \tfrac{1}{\lambda_{1,n}^2} \sqrt{\tfrac{2}{L}} \int_x^L \beta(s) \sin\big(\tfrac{n\pi}{L}s\big) \sinh(\lambda_{1,n} (x-s)) \,\mathrm{d}s  \\
& + \tfrac{1}{\lambda_{1,n}} \sqrt{\tfrac{2}{L}} \int_0^x \int_\tau^L \beta(s) \sin\big(\tfrac{n\pi}{L}s\big) \,\mathrm{d}s\,\mathrm{d}\tau , \\ 
\psi^3_{1,n}(x) = & \tfrac{\gamma_n}{\lambda_{1,n} \big( \cosh(\lambda_{1,n} L) + \alpha \sinh(\lambda_{1,n} L) \big) } \sqrt{\tfrac{2}{L}} \big( \cosh(\lambda_{1,n} (x-L)) - \alpha \sinh(\lambda_{1,n} (x-L)) \big) \\
& + \tfrac{1}{\lambda_{1,n}} \sqrt{\tfrac{2}{L}} \int_x^L \beta(s) \sin\big(\tfrac{n\pi}{L}s\big) \sinh(\lambda_{1,n} (x-s)) \,\mathrm{d}s ,
\end{align*}
with
\begin{subequations}\label{def_gamma_n}
    \begin{equation}\label{eq: def gamma_1 - 1}
        \gamma_n = \int_0^L \beta(s) \sin\big(\tfrac{n\pi}{L}s\big) \sinh(\lambda_{1,n} s) \,\mathrm{d}s
    \end{equation}
    whenever $\tfrac{n^2 \pi^2}{L^2}\neq c$, i.e., when $\lambda_{1,n} \neq 0$, and 
    \begin{align*}
        \psi^1_{1,n}(x) = & \sqrt{\tfrac{2}{L}} 		\sin\big(\tfrac{n\pi}{L}x\big) , \quad
		\psi^2_{1,n}(x) =  \alpha\gamma_n \sqrt{\tfrac{2}{L}} x , \\ 
		\psi^3_{1,n}(x) = & \sqrt{\tfrac{2}{L}} \int_0^x \int_\tau^L \beta(s) \sin\big(\tfrac{n\pi}{L}s\big) \,\mathrm{d}s\,\mathrm{d}\tau ,
    \end{align*}
    with 
    \begin{equation}\label{eq: def gamma_1 - 2}
        \gamma_n = \int_0^L \int_\tau^L \beta(s) \sin\big(\tfrac{n\pi}{L}s\big) \,\mathrm{d}s\,\mathrm{d}\tau
    \end{equation}
\end{subequations}
whenever $\tfrac{n^2 \pi^2}{L^2} = c$, i.e., when $\lambda_{1,n} = 0$, and
\begin{align*}
& \psi^1_{2,m}(x) = 0 , \quad
\psi^2_{2,m}(x) = \tfrac{A_m}{L (\bar{\lambda}_{2,m})^2} \sinh(\bar{\lambda}_{2,m} x) , \\
& \psi^3_{2,m}(x) = - \tfrac{A_m}{L \bar{\lambda}_{2,m}} \sinh(\bar{\lambda}_{2,m} x) .
\end{align*}
The eigenvectors satisfy $\mathcal{A}^* \psi_{1,n} = \lambda_{1,n}\psi_{1,n}$ and $\mathcal{A}^* \psi_{2,m} = \bar{\lambda}_{2,m}\psi_{2,m}$ and have been normalized so that $\langle \phi_{1,n} , \psi_{1,n} \rangle = 1$ and $\langle \phi_{2,m} , \psi_{2,m} \rangle = 1$ for all $n\in\mathbb{N}^*$ and $m\in\mathbb{Z}$.
\end{lemma}


\begin{proof}
Using \eqref{def_A*}, it can be checked that $\langle \mathcal{A}(f_1,g_1,h_1) , (f_2,g_2,h_2) \rangle = \langle (f_1,g_1,h_1) , \mathcal{A}^* (f_2,g_2,h_2) \rangle$ for all $(f_1,g_1,h_1) \in D(\mathcal{A})$ and all $(f_2,g_2,h_2) \in D(\mathcal{A}^*)$. Let us now solve $\mathcal{A}^*(f,g,h) = \lambda (f,g,h)$ for some $\lambda\in\mathbb{C}$ and $(f,g,h)\in D(\mathcal{A}^*)$. This gives 
$f'' + c f  = \lambda f$, $- g''  = \lambda h$ and 
$- \int_0^{x} \int_L^\tau \beta(s) f(s) \,\mathrm{d}s\,\mathrm{d}\tau - h  = \lambda g $.
Proceeding similarly to the proof of Lemma~\ref{lem: eigenstructures of A}, the case $f = 0$ gives $\lambda = \lambda_{2,-m}$ for some $m\in\mathbb{Z}$ with $(f,g,h) = \psi_{2,m}$ while the case $f \neq 0$ gives $\lambda = \lambda_{1,n}$ for some $n\in\mathbb{N}^*$. Studying separatly the cases $\lambda_{1,n} = 0$ and $\lambda_{1,n} \neq 0$, we infer the claimed conclusion. Note that in the case $\lambda_{1,n} \neq 0$ the introduced eigenfunctions are well-defined because it follows from \eqref{eq: condition for selecting alpha} that $\cosh(\lambda_{1,n} L) + \alpha \sinh(\lambda_{1,n} L) \neq 0$.
\end{proof}


Lemma~\ref{lem: A riesz spectral} shows that $\mathcal{A}$ is a Riesz operator. This will allow us in the next sections to perform stabilization in $\mathcal{H}^0$ norm. But actually, the stabilization of the parabolic part of the system can also be achieved in $H^1$ norm (compared to the $L^2$ norm in the case of the space $\mathcal{H}^0$) thanks to the following result.


\begin{lemma}\label{lemma: Riesz basis H1}
	$\Phi^1 = \{\tfrac{L}{n\pi}\phi_{1,n}\, \mid\, n\in\mathbb{N}^* \} \cup \{\phi_{2,m}\, \mid\, m \in\mathbb{Z} \}$ is a Riesz basis of the Hilbert space
	\begin{subequations}\label{eq: space H1}
	\begin{equation}
		\mathcal{H}^1 = H_0^1(0,L) \times H_{(0)}^1(0,L) \times L^2(0,L)
	\end{equation}
	endowed with the inner product 
	\begin{equation}
        \langle (f_1,g_1,h_1) , (f_2,g_2,h_2) \rangle = \int_0^L (f_1'\overline{f_2'} + g'_1 \overline{g'_2} + h_1\overline{h_2} ) .
    \end{equation}
\end{subequations}
\end{lemma}

\begin{proof}
Denoting by $\triangle:H^1_0(0,L) \cap H^2(0,L)\rightarrow L^2(0,L)$ the usual Dirichlet operator, the operator
$J  :  \mathcal{H}^1 \rightarrow \mathcal{H}^0$ defined by $J(f,g,h)=(\sqrt{-\triangle}f , g ,h)$ is easily seen to be a surjective isometry. Indeed, by integration by parts,
\begin{align*}
& \langle J(f_1,g_1,h_1),J(f_2,g_2,h_2)\rangle
= \langle(-\triangle f_1,g_1,h_1),(f_2,g_2,h_2)\rangle \\
& \qquad = \int_0^L f_1' f_2' + g_1' g_2' + h_1 h_2 \,\mathrm{d}x 
= \langle(f_1,g_1,h_1),(f_2,g_2,h_2)\rangle .
\end{align*} 
Hence $\{J^{-1}\phi_{1,n}\, \mid\, n\in\mathbb{N}^* \} \cup \{J^{-1}\phi_{2,m}\, \mid\, m \in\mathbb{Z} \}$ is a Riesz basis of $\mathcal{H}^1$, and we have $J^{-1}\phi_{1,n} = \tfrac{L}{n\pi} \phi_{1,n}$ and $J^{-1}\phi_{2,m} = ( (\sqrt{-\triangle})^{-1} \phi_{2,m}^1 , \phi_{2,m}^2 , \phi_{2,m}^3 )$. To use the Bari theorem, we have to show that
$
    \sum_{m\in\mathbb{Z}} \Vert \phi^1_{2,m} - (\sqrt{-\triangle})^{-1}\phi^1_{2,m} \Vert_{H_0^1}^2
    < +\infty .
$
Since $\Vert (\sqrt{-\triangle})^{-1}\phi^1_{2,m} \Vert_{H_0^1} = \Vert \phi^1_{2,m} \Vert_{L^2}$, it is sufficient to show that $\sum_{m\in\mathbb{Z}} \Vert \phi^1_{2,m} \Vert_{H_0^1}^2 < +\infty$. Following the proof of Lemma~\ref{lem: A riesz spectral}, we have $\phi_{2,m}^1 = f_{m,1} + f_{m_2}$ with $\Vert f_{m,1}' \Vert_{L^\infty} = O(1/\vert m \vert^{3/2})$ and $\Vert f_{m,2}' \Vert_{L^\infty} = O(1/\vert m \vert^{3/2})$. 
\end{proof}

\section{Spectral reduction}\label{sec: spectral reduction}

\subsection{Homogeneous representation and spectral reduction}

\label{sub sec: Finite-dimensional reduced model}
\noindent 
In order to work with a homogeneous representation of \eqref{eq: cascade equation - premilinary feedback}, we lift the control from the boundary into the domain (see e.g. \cite[Sec.~3.3]{curtain2012introduction}) by making the change of variables:
\begin{subequations}\label{eq: change of variable}
\begin{align}
& w^1(t,x) = y(t,x) , \quad w^2(t,x) = z(t,x) , \\
& w^3(t,x) = \partial_t z(t,x) - \tfrac{x}{\alpha L} v(t) .
\end{align}
\end{subequations}
Assuming that $v$ is of class $\mathcal{C}^1$ (this assumption will be fulfilled by the upcoming adopted feedback control strategy), the control system becomes
\begin{subequations}\label{eq: cascade equation - homonegenous}
\begin{align}
& \partial_t w^1 = \partial_{xx} w^1 + c w^1 + \beta w^2 , \\
& \partial_t w^2 = w^3 + \tfrac{x}{\alpha L} v , \\
& \partial_t w^3 = \partial_{xx} w^2 - \tfrac{x}{\alpha L} \dot{v} , \\
& w^1(t,0) = w^1(t,L) = 0 , \\ 
& w^2(t,0) = 0 , \qquad \partial_x w^2(t,L) + \alpha w^3(t,L) = 0 ,\\
& w^1(0,x) = y_0(x) , \qquad w^2(0,x) = z_0(x) , \\ 
& w^3(0,x) = z_1(x) - \tfrac{x}{\alpha L} v(0) .
\end{align}
\end{subequations}
Setting $W(t)=\big(w^1(t,\cdot),w^2(t,\cdot),w^3(t,\cdot)\big)$, $a(x) = (0,\tfrac{x}{\alpha L},0) \in\mathcal{H}^0$ and $b(x) = (0,0,-\tfrac{x}{\alpha L}) \in\mathcal{H}^0$, one has
\begin{subequations}\label{eq: cascade equation - homonegenous abstract}
\begin{align}
\dot{W}(t) & = \mathcal{A} W(t) + a v(t) + b \dot v(t) ,\\
W(0,x) & = \big( y_0(x) , z_0(x) , z_1(x) - \tfrac{x}{\alpha L} v(0) \big) .
\end{align}
\end{subequations}
We now use the Riesz basis $\Phi$ and $\Psi$ defined in Section~\ref{sec_prelim_A}, expanding
\begin{equation}\label{eq: projection system trajectory into Riesz basis}
W(t,\cdot) = \sum_{n\in\mathbb{N}^*} w_{1,n}(t) \phi_{1,n} + \sum_{m \in\mathbb{Z}} w_{2,m}(t) \phi_{2,m} 
\end{equation}
with $w_{1,n} = \langle W(t,\cdot) , \psi_{1,n} \rangle$, $w_{2,m} = \langle W(t,\cdot) , \psi_{2,m} \rangle$, $a_{1,n} = \langle a , \psi_{1,n} \rangle$, $a_{2,m} = \langle a , \psi_{2,m} \rangle$, $b_{1,n} = \langle b , \psi_{1,n} \rangle$ and $b_{2,m} = \langle b , \psi_{2,m} \rangle$ for all $n\in\mathbb{N}^*$ and $m \in\mathbb{Z}$. 
Defining the new control $v_d = \dot{v}$ (auxiliary input for control design), we thus have
\begin{subequations}\label{eq: spectral reduction}
\begin{align}
\dot{w}_{1,n} & = \lambda_{1,n} w_{1,n} + a_{1,n} v + b_{1,n} v_d , \quad n\in\mathbb{N}^* , \label{eq: spectral reductiona} \\
\dot{w}_{2,m} & = \lambda_{2,m} w_{2,m} + a_{2,m} v + b_{2,m} v_d , \quad m \in\mathbb{Z}, \label{eq: spectral reductionb}
\end{align}
\end{subequations}
\begin{equation}\label{eq: inegral action}
\dot{v} = v_d . \qquad\qquad\qquad\qquad\qquad\qquad\qquad\quad\ 
\end{equation}

\subsection{Finite-dimensional model and controllability properties}\label{subsec: finite-dim model}

\subsubsection{Finite-dimensional model}

For a given $\delta >0$, let us choose an integer $N_0 \in\mathbb{N}^*$, large enough so that
$$
\lambda_{1,n} \leq \lambda_{1,N_0+1} = c - \tfrac{(N_0+1)^2\pi^2}{L^2} < -\delta < 0 \quad\forall n \geq N_0 +1.
$$
We next consider the finite-dimensional system consisting of the $N_0$ first modes of the plant \eqref{eq: spectral reductiona} associated with the eigenvalues $\lambda_{1,n}$ of the reaction-diffusion equation. Setting
\begin{align*}
W_0 & = \begin{pmatrix}
w_{1,1} & w_{1,2} & \ldots & w_{1,N_0}
\end{pmatrix}^\top \in\mathbb{R}^{N_0} , \\
A_0 & = \mathrm{diag}\big(
\lambda_{1,1} , \lambda_{1,2} , \ldots , \lambda_{1,N_0}
\big) \in\mathbb{R}^{N_0 \times N_0} , \\
B_{a,0} & = \begin{pmatrix}
a_{1,1} & a_{1,2} & \ldots & a_{1,N_0}
\end{pmatrix}^\top \in\mathbb{R}^{N_0} , \\
B_{b,0} & = \begin{pmatrix}
b_{1,1} & b_{1,2} & \ldots & b_{1,N_0}
\end{pmatrix}^\top \in\mathbb{R}^{N_0} ,
\end{align*}
we have 
\begin{equation}\label{truncature:eq}
\dot{W}_0(t) = A_0 W_0(t) + B_{a,0} v(t) + B_{b,0} v_d(t) .
\end{equation}
Note in particular that these dynamics do not capture any mode $\lambda_{2,m}$ of the hyperbolic part of the plant as described by \eqref{eq: spectral reductionb}. Augmenting the state vector and the matrices by setting
\begin{equation*}
W_1 = \begin{pmatrix} v \\ W_0 \end{pmatrix} , \quad
A_1 = \begin{pmatrix} 0 & 0 \\ B_{a,0} & A_0 \end{pmatrix} , \quad
B_1 = \begin{pmatrix} 1 \\ B_{b,0} \end{pmatrix} ,
\end{equation*}
we deduce that
\begin{equation}\label{truncature:eq:aug}
\dot{W}_1(t) = A_1 W_1(t) + B_1 v_d(t) .
\end{equation}

\subsubsection{Controllability}

Recalling that $v_d$ is viewed as an auxiliary control input, the following result characterizes the controllability property of \eqref{truncature:eq:aug}.


\begin{lemma}\label{kalman:contro:lemma}
The pair $(A_1,B_1)$ satisfies the Kalman condition if and only if $\gamma_n \neq 0$ for any $n\in\{1,\ldots,N_0\}$, where $\gamma_n$ is defined by \eqref{eq: def gamma_1 - 1} if $\tfrac{n^2 \pi^2}{L^2}\neq c$ and by \eqref{eq: def gamma_1 - 2} if $\tfrac{n^2 \pi^2}{L^2} = c$.
\end{lemma}


\begin{proof}
By the Hautus test, $(A_1,B_1)$ does not satisfy the Kalman condition if and only if there exist $\lambda\in\mathbb{C}$ and $(x_1,x_2)\in\mathbb{C}\times\mathbb{C}^{N_0}\setminus\{(0,0)\}$ such that $\overline{x_2}^\top B_{a,0} = \lambda \overline{x_1}^\top$, $\overline{x_2}^\top A_0 = \lambda \overline{x_2}^\top$, and $\overline{x_1}^\top + \overline{x_2}^\top B_{b,0} = 0$. This is possible if and only if there exists $x_2 \neq 0$ such that $A_0^* x_2 = \overline{\lambda} x_2$ and $\overline{x_2}^\top \big( B_{a,0} + \lambda B_{b,0} \big) = 0$, i.e., if and only if $\lambda\in\Lambda=\{\lambda_{1,1} , \lambda_{1,2} , \ldots , \lambda_{1,N_0} \}$ while (because the eigenvalues of $\mathcal{A}^*$ are simple) there exists $z\in D(\mathcal{A}^*)\setminus\{0\}$ such that $\mathcal{A}^*z=\overline{\lambda}z$ and $\langle a + \lambda b , z \rangle = 0$. Setting $z=(z^1,z^2,z^3)$, the equation $\mathcal{A}^*z=\overline{\lambda}z$ gives
\begin{subequations}\label{eq: proof Kalman condition - sys eq to be solved - 1-5}
\begin{align}
& (z^1)'' + c z^1 = \overline{\lambda} z^1 \label{eq: proof Kalman condition - sys eq to be solved - 1} \\
& - z^3 + \int_0^{(\cdot)} \int_\tau^L \beta(s) z^1(s) \,\mathrm{d}s\,\mathrm{d}\tau = \overline{\lambda} z^2 \label{eq: proof Kalman condition - sys eq to be solved - 2} \\
& - (z^2)'' = \overline{\lambda} z^3 \label{eq: proof Kalman condition - sys eq to be solved - 3} \\
& z^1(0)=z^1(L)=z^2(0)=z^3(0)=0 \label{eq: proof Kalman condition - sys eq to be solved - 4} \\ 
& (z^2)'(L) - \alpha z^3(L) = 0 \label{eq: proof Kalman condition - sys eq to be solved - 5}
\end{align}
\end{subequations}
while 
\begin{align*}
0 
& = \langle a + \lambda b , z \rangle 
= \int_0^L \big( \big(\tfrac{x}{\alpha L}\big)' \overline{(z^2)'(x)} - \tfrac{\lambda x}{\alpha L} \overline{z^3(x)} \big) \mathrm{d}x \\
& = \big[ \tfrac{x}{\alpha L} \overline{(z^2)'(x)} \big]_{x=0}^{x=L} - \int_0^L \tfrac{x}{\alpha L} \overline{\big( (z^2)''(x) + \overline{\lambda} z^3(x) \big)} \,\mathrm{d}x \\
& = \tfrac{1}{\alpha} \overline{(z^2)'(L)}
\end{align*}
where we have used \eqref{eq: proof Kalman condition - sys eq to be solved - 3}. Combining with \eqref{eq: proof Kalman condition - sys eq to be solved - 5}, we get
\begin{equation}\label{eq: proof Kalman condition - sys eq to be solved - 6}
(z^2)'(L) = z^3(L) = 0 .
\end{equation}

Let us discard the case $z^1 = 0$ by noting that, in this case, \eqref{eq: proof Kalman condition - sys eq to be solved - 2} gives $- z^3 = \overline{\lambda} z^2$. Combining this result with \eqref{eq: proof Kalman condition - sys eq to be solved - 3} and \eqref{eq: proof Kalman condition - sys eq to be solved - 6} we infer that $(z^3)'' - (\overline{\lambda})^2 z^3 = 0$ with $z^3(L)=(z^3)'(L)=0$. By Cauchy uniqueness, we deduce that $z^3 = 0$. Finally, \eqref{eq: proof Kalman condition - sys eq to be solved - 3} gives $(z^2)'' = 0$ with $z^2(0)=(z^2)'(L) = 0$, hence $z^2 =0$. This is a contradiction with the initial assumption that $z \neq 0$.

Hence, we must have $z^1 \neq 0$. Based on \eqref{eq: proof Kalman condition - sys eq to be solved - 1} and \eqref{eq: proof Kalman condition - sys eq to be solved - 4} we have $(z^1)'' = \overline{(\lambda-c)} z^1$ with $z^1(0)=z^1(L)=0$. Then there exist $A \neq 0$ and $n \in\mathbb{N}^*$ such that $z^1(x) = A\sin\big(\tfrac{n\pi}{L}x\big)$ and $\lambda = c - \tfrac{n^2 \pi^2}{L^2}$. Since $\lambda\in\Lambda$, we have $\lambda = \lambda_{1,n}$ for some $n\in\{1,\ldots,N_0\}$. We infer from \eqref{eq: proof Kalman condition - sys eq to be solved - 2}, \eqref{eq: proof Kalman condition - sys eq to be solved - 3}, and \eqref{eq: proof Kalman condition - sys eq to be solved - 6} that
\begin{subequations}
\begin{align}
& (z^3)''(x) - \lambda_{1,n}^2 z^3(x) = - A \beta(x) \sin\big(\tfrac{n\pi}{L}x\big) , \label{eq: proof Kalman condition - sys eq to be solved - 7} \\
& z^3(0) = z^3(L) = (z^3)'(L) = 0 , \label{eq: proof Kalman condition - sys eq to be solved - 8}
\end{align}
\end{subequations}
where we recall that $\lambda_{1,n} = c - \tfrac{n^2 \pi^2}{L^2}$.

Assume first that $\lambda_{1,n} \neq 0$. In this case, integrating \eqref{eq: proof Kalman condition - sys eq to be solved - 7}, there exists $\delta_1,\delta_2\in\mathbb{R}$ such that
\begin{align*}
z^3(x) = & \Big( \delta_1 + \tfrac{A}{2\lambda_{1,n}} \int_0^x e^{\lambda_{1,n} s} \beta(s) \sin\big( \tfrac{n \pi}{L} s \big) \,\mathrm{d}s \Big) e^{-\lambda_{1,n} x} \\
& + \Big( \delta_2 - \tfrac{A}{2\lambda_{1,n}} \int_0^x e^{-\lambda_{1,n} s} \beta(s) \sin\big( \tfrac{n \pi}{L} s \big) \,\mathrm{d}s \Big) e^{\lambda_{1,n} x} .
\end{align*}
The condition $z^3(0) = 0$ gives $\delta_1 + \delta_2 = 0$, hence 
\begin{equation*}
z^3(x) = - 2 \delta_1 \sinh(\lambda_{1,n} x) + \tfrac{A}{\lambda_{1,n}} \int_0^x \beta(s)  \sinh(\lambda_{1,n}(s-x)) \sin\big( \tfrac{n \pi}{L} s \big) \,\mathrm{d}s .
\end{equation*}
The conditions $z^3(L) = 0$ and $(z^3)'(L) = 0$ give the system
\begin{equation*}
\begin{pmatrix}
\sinh(\lambda_{1,n} L) & \chi_1 \\
\cosh(\lambda_{1,n} L) & \chi_2
\end{pmatrix}
\begin{pmatrix}
-2 \delta_1 \lambda_{1,n} \\ A
\end{pmatrix}
= 0 
\end{equation*}
where $\chi_1 = \int_0^L \beta(s) \sinh(\lambda_{1,n}(s-L)) \sin\big( \tfrac{n \pi}{L} s \big) \,\mathrm{d}s$ and $\chi_2 = -\int_0^L \beta(s) \cosh(\lambda_{1,n}(s-L)) \sin\big( \tfrac{n \pi}{L} s \big) \,\mathrm{d}s$. Since $A \neq 0$ and $\lambda_{1,n} \neq 0$, the determinant of the above $2 \times 2$ matrix must be zero, i.e., 
\begin{align*}
0 = \int_0^L \beta(s) \sinh(\lambda_{1,n} s) \sin\big( \tfrac{n \pi}{L} s \big) \,\mathrm{d}s = \gamma_n .
\end{align*}
Under this condition, one can compute $A \neq 0$ and $\delta_1$, and thus obtain $z^3$. Finally, $z^2$ is obtained by integrating twice \eqref{eq: proof Kalman condition - sys eq to be solved - 3} and using the conditions $z^2(0)=(z^2)'(L)=0$ borrowed from \eqref{eq: proof Kalman condition - sys eq to be solved - 4} and \eqref{eq: proof Kalman condition - sys eq to be solved - 6}. The obtained $z=(z^1,z^2,z^3) \neq 0$ satisfies \eqref{eq: proof Kalman condition - sys eq to be solved - 1-5} and \eqref{eq: proof Kalman condition - sys eq to be solved - 6}.

Assume now that $\lambda_{1,n}=0$. In this case, \eqref{eq: proof Kalman condition - sys eq to be solved - 7} reduces to $(z^3)'' = - A \beta(x) \sin\big(\tfrac{n\pi}{L}x\big)$. Owing to \eqref{eq: proof Kalman condition - sys eq to be solved - 8}, we infer from the first and third conditions that $z^3(x)= A \int_0^x \int_\tau^L \beta(s) \sin\big(\tfrac{n\pi}{L}s\big) \,\mathrm{d}s\,\mathrm{d}\tau$ while the second condition gives 
$$0 = \int_0^L \int_\tau^L \beta(s) \sin\big(\tfrac{n\pi}{L}s\big) \,\mathrm{d}s\,\mathrm{d}\tau = \gamma_n$$
because $A \neq 0$. As previously, we then obtain $z^2$ by integrating twice \eqref{eq: proof Kalman condition - sys eq to be solved - 3} with the conditions $z^2(0)=(z^2)'(L)=0$ borrowed from \eqref{eq: proof Kalman condition - sys eq to be solved - 4} and \eqref{eq: proof Kalman condition - sys eq to be solved - 6}. Hence, the computed $z=(z^1,z^2,z^3) \neq 0$ satisfies \eqref{eq: proof Kalman condition - sys eq to be solved - 1-5} and \eqref{eq: proof Kalman condition - sys eq to be solved - 6}.
\end{proof}

\begin{remark}\label{rem_modes}
    The result of Lemma~\ref{kalman:contro:lemma} can be interpreted as follows. Since $X(t)=\big(y(t,\cdot),z(t,\cdot),\partial_t z(t,\cdot)\big)^\top$, we have $W = X + b v$. Hence
    \begin{subequations}\label{eq: spectral reduction bis}
    \begin{align}
    \dot{x}_{1,n} & = \lambda_{1,n} x_{1,n} + \beta_{1,n} v , \qquad n\in\mathbb{N}^*,\\
    \dot{x}_{2,m} & = \lambda_{2,m} x_{2,m} + \beta_{2,m} v , \qquad m \in\mathbb{Z},
    \end{align}
    \end{subequations}
    where $x_{1,n}(t) = \langle X(t,\cdot) , \psi_{1,n} \rangle$, $x_{2,m}(t) = \langle X(t,\cdot) , \psi_{2,m} \rangle$, 
    $\beta_{1,n} = a_{1,n} + \lambda_{1,n} b_{1,n}$ and $\beta_{2,m} = a_{2,m} + \lambda_{2,m} b_{2,m}$. The computations done in the proof of Lemma~\ref{kalman:contro:lemma} show that $\beta_{i,l} = \tfrac{1}{\alpha} \overline{(\psi_{i,l}^2)'(L)} = \overline{\psi_{i,l}^3(L)}$. This means that the control has no impact on the mode $\lambda_{i,l}$ if and only if $\psi_{i,l}^3(L) = 0$.
\end{remark}

\subsection{Discussion on the coefficients $\gamma_n$}

We discuss the behavior of the coefficients $\gamma_n$ for particular examples of functions $\beta$. Let us assume, in the present section, that
\begin{equation}\label{beta_char}
\beta(x) = \beta_0 \, \mathds{1}_{[a,b]}(x) \qquad \forall x\in(0,L)
\end{equation}
for some $\beta_0 \in\mathbb{R}\setminus\{0\}$ and $0 \leq a < b \leq L$. 
Recalling that $\lambda_{1,n} = c - \tfrac{n^2 \pi^2}{L^2}$, we infer from \eqref{eq: def gamma_1 - 1} that
\begin{subequations}
\label{gamma:def:remark}
\begin{align}
\gamma_n 
= & \tfrac{\beta_0}{\lambda_{1,n}^2 + \tfrac{n^2 \pi^2}{L^2}} \Big( - \tfrac{n\pi}{L} \sinh( \lambda_{1,n} b ) \cos\big( \tfrac{n\pi b}{L} \big) + \tfrac{n\pi}{L} \sinh( \lambda_{1,n} a ) \cos\big( \tfrac{n\pi a}{L} \big)  \nonumber \\
& \hspace{3cm} + \lambda_{1,n} \cosh( \lambda_{1,n} b ) \sin\big( \tfrac{n\pi b}{L} \big) - \lambda_{1,n} \cosh ( \lambda_{1,n} a ) \sin\big( \tfrac{n\pi a}{L} \big) \Big)  
\end{align}
if $\tfrac{n^2 \pi^2}{L^2}\neq c$, and
\begin{equation}
\gamma_n = - \tfrac{\beta_0 L}{n\pi} \big( b \cos\big(\tfrac{n\pi}{L} b \big) - a \cos\big(\tfrac{n\pi}{L} a \big) \big) + \tfrac{\beta_0 L^2}{n^2\pi^2} \big( \sin\big(\tfrac{n\pi}{L} b \big) - \sin\big(\tfrac{n\pi}{L} a \big) \big) 
\end{equation}
\end{subequations}
if $\tfrac{n^2 \pi^2}{L^2} = c$.
There are some critical values of $a,b$ for which there exists $n\in\mathbb{N}^*$ such that $\gamma_n=0$ and thus controllability is lost for the mode $n$ (see Lemma \ref{kalman:contro:lemma}).
Let us elaborate further.

In the particular case $a=0$ and $b=L$, i.e., $\beta\equiv\beta_0$ on $(0,L)$, we get from \eqref{gamma:def:remark} that
$$
\gamma_n = \tfrac{(-1)^{n}\beta_0 \tfrac{n\pi}{L}}{\big( \tfrac{n^2 \pi^2}{L^2} - c \big)^2 + \tfrac{n^2 \pi^2}{L^2}} \sinh\big(\big( \tfrac{n^2 \pi^2}{L^2} - c \big) L \big) \neq 0 
$$
if $\tfrac{n^2 \pi^2}{L^2}\neq c$, and $\gamma_n = \tfrac{(-1)^{n+1}\beta_0 L^2}{n\pi} \neq 0$ if $\tfrac{n^2 \pi^2}{L^2} = c$.
Hence, the result of Lemma~\ref{kalman:contro:lemma} ensures that the finite-dimensional system \eqref{truncature:eq:aug} is always controllable, whatever the numbers $N_0$ of modes of the parabolic spectrum captured by the dynamics. 

Consider now the case $L=1$, $c = 50$, $\beta_0 = 1$, and $a = 0$. Figure~\ref{fig: gamma_2_for_L_1_c_50_a_0} shows $\gamma_2$ as a function of $b \in [0,L]$; $\gamma_2$ vanishes at $b\approx 0.586$ and thus, controllability is lost for this value of $b$. 
\begin{figure}[h]
	\centering
	\includegraphics[width=3.5in]{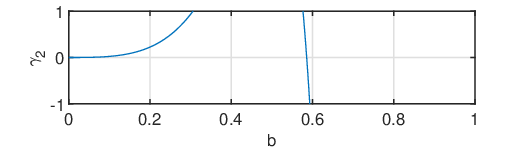}
	\caption{$\gamma_2$ as a function of $b \in [0,1]$, with $L=1$, $c=50$, $\beta_0 = 1$, and $a=0$.}
	\label{fig: gamma_2_for_L_1_c_50_a_0}
\end{figure}

The lext lemma shows that loss of controllability rarely happens.
\begin{lemma}\label{prop:contr}
Define $S=\{(a,b)\,\mid\, 0\leq a<b\leq L\}$ and recall $\gamma_n$ is defined by \eqref{def_gamma_n}. The subset $\hat S$ of $S$ such that $\gamma_n\neq 0$ for any $n\in\mathbb{N}^*$ is dense and of full Lebesgue measure in $S$.
\end{lemma}


\begin{proof}
Given any $n\in\mathbb{N}^*$, the set $Z_n=\{(a,b)\in S\,\mid\, \gamma_n=0\}$ is closed, of zero measure and of empty interior because $\gamma_n$ is a nontrivial analytic function of $(a,b)$. 
The union $Z$ of all $Z_n$, $n\in\mathbb{N}^*$, is of zero measure, and of empty interior by the Baire theorem.
The set $\hat S$ is then defined as the complement of $Z$ in $S$.
\end{proof}


\begin{remark}\label{rem_nonrobust}
Lemmas~\ref{kalman:contro:lemma} and \ref{prop:contr} imply that, when $(a,b)\in \hat S$, the finite-dimensional system \eqref{truncature:eq:aug} is always controllable, whatever the numbers $N_0$ of modes of the parabolic spectrum captured by the dynamics. However, although $\hat S$ is dense and of full Lebesgue measure in $S$, it may fail to be open, and thus robustness of controllability with respect to $(a,b)$ may fail.
\end{remark}

\subsection{Measurement}

In view of the control design, based on \eqref{eq: projection system trajectory into Riesz basis}, we note that the system output $y_o(t)$ defined by \eqref{eq: measurement heat distributed} is expanded as
\begin{align*}
y_o(t) & = \int_0^L c_o(x) y(t,x) \,\mathrm{d}x = \int_0^L c_o(x) w^1(t,x) \,\mathrm{d}x \\
& = \sum_{n\in\mathbb{N}^*} c_{1,n} w_{1,n}(t) + \sum_{m \in \mathbb{Z}} c_{2,m} w_{2,m}(t) \\
&= C_0 W_0(t) + \sum_{n \geq N_0 + 1} c_{1,n} w_{1,n}(t) + \sum_{m \in \mathbb{Z}} c_{2,m} w_{2,m}(t)
\end{align*}
with 
\begin{subequations}\label{eq: def c_{i,k}}
\begin{align}
c_{1,n} & = \int_0^L c_o(x) \phi^1_{1,n}(x) \,\mathrm{d}x ,\quad n\in\mathbb{N}^*,  \\
c_{2,m} & = \int_0^L c_o(x) \phi^1_{2,m}(x) \,\mathrm{d}x ,\quad m \in \mathbb{Z}.
\end{align}
\end{subequations}
and
$C_0 = \begin{pmatrix}
c_{1,1} & c_{1,2} & \ldots & c_{1,N_0}
\end{pmatrix} \in\mathbb{R}^{1 \times N_0}$.
Since the matrix $A_0$ is diagonal with simple eigenvalues, we have the following result which will be instrumental for the proposed output-feedback control strategy.


\begin{lemma}\label{lem: obs for cont}
The pair $(A_0,C_0)$ satisfies the Kalman condition if and only if 
$$
c_{1,n} =  \sqrt{\tfrac{2}{L}} \int_0^L \!\! c_o(x) \sin\big( \tfrac{n\pi}{L} x \big) \,\mathrm{d}x \neq 0 , \;\; \forall n\in\{1,\ldots,N_0\}. 
$$ 
\end{lemma}

\section{Feedback stabilization}\label{sec: control design}

\noindent
Given two integers $N \geq N_0 +1$ and $M \in\mathbb{N}$ to be chosen later, we define the following control strategy:
\begin{subequations}\label{eq: controller}
\begin{align}
\dot{\hat{w}}_{1,n} & = \lambda_{1,n} \hat{w}_{1,n} + a_{1,n} v + b_{1,n} v_d \nonumber \\
& \phantom{=}\; - l_{1,n} \bigg( \sum_{k=1}^N c_{1,k} \hat{w}_{1,k} + \sum_{\vert l \vert \leq M} c_{2,l} \hat{w}_{2,l} - y_o \bigg) , \qquad 1 \leq n \leq N_0 , \\
\dot{\hat{w}}_{1,n} & = \lambda_{1,n} \hat{w}_{1,n} + a_{1,n} v + b_{1,n} v_d , \qquad N_0 + 1 \leq n \leq N , \\
\dot{\hat{w}}_{2,m} & = \lambda_{2,m} \hat{w}_{2,m} + a_{2,m} v + b_{2,m} v_d , \qquad \vert m \vert \leq M , \\
v_d & = k_v v + \sum_{n=1}^{N_0} k_{1,n} \hat{w}_{1,n} ,
\end{align}
\end{subequations}
where $k_v$ and $k_{1,n}$, for $1\leq n\leq N_0$, are the feedback gains and $l_{1,n}$, for $1\leq n\leq N_0$, are the observer gains. This finite-dimensional control strategy leveraging a Luenberger-type observer on a finite number of modes is inspired by the seminal work~\cite{sakawa1983feedback} and its more recent developments~\cite{katz2020constructive, lhachemi2020finite, lhachemi2021nonlinear, grune2021finite}. We set
\begin{align*}
K & = \begin{pmatrix}
k_v & k_{1,1} & k_{1,2} & \ldots & k_{1,N_0}
\end{pmatrix} \in\mathbb{R}^{1 \times (N_0+1)} , \\
L & = \begin{pmatrix}
l_{1,1} & l_{1,2} & \ldots & l_{1,N_0}
\end{pmatrix}^\top \in\mathbb{R}^{N_0} .
\end{align*}

We now state the main result of this paper.

\begin{theorem}\label{thm: main result 2}
Let $\delta > 0$ be arbitrary. Let $N_0\in\mathbb{N}^*$ and $\alpha > 1$ be such that $\lambda_{1,N_0+1} < - \delta$ and $\rho = \tfrac{1}{2L} \log\big(\tfrac{\alpha-1}{\alpha+1}\big) = \operatorname{Re}\lambda_{2,m} < - \delta$. Assume that:
\begin{itemize}
\item $\gamma_n \neq 0$ for any $n\in\{1,\ldots,N_0\}$, where $\gamma_n$ is defined by \eqref{eq: def gamma_1 - 1} if $\tfrac{n^2 \pi^2}{L^2}\neq c$ and by \eqref{eq: def gamma_1 - 2} if $\tfrac{n^2 \pi^2}{L^2} = c$;
\item $c_{1,n} \neq 0$ for any $n\in\{1,\ldots,N_0\}$, where $c_{1,n}$ is defined by \eqref{eq: def c_{i,k}}. 
\end{itemize}
Let $K\in\mathbb{R}^{1 \times (N_0 +1)}$ and $L\in\mathbb{R}^{N_0}$ be such that $A_1 + B_1 K$ and $A_0 - L C_0$ are Hurwitz with eigenvalues of real part less than $-\delta < 0$. 

Then, for all integers $N \geq N_0 +1$ and $M$ sufficiently large\footnote{They must be chosen large enough so that the inequalities \eqref{constraints_NM} are satisfied.},
there exists $C > 0$ such that 
any solution of the system \eqref{eq: cascade equation} in closed-loop with the output-feedback control \eqref{eq: measurement heat distributed}, \eqref{eq: preliminary feedback}, \eqref{eq: inegral action}, \eqref{eq: controller} satisfies
\begin{align}
& \big\Vert ( y(t,\cdot) , z(t,\cdot) , \partial_t z(t,\cdot) ) \big\Vert_{H} + \vert v(t) \vert + \sum_{n=1}^N \vert \hat{w}_{1,n}(t) \vert + \sum_{\vert m \vert \leq M} \vert \hat{w}_{2,m}(t) \vert \nonumber \\
& \leq
C e^{-\delta t}
\Big(
\big\Vert ( y(0,\cdot) , z(0,\cdot) , \partial_t z(0,\cdot) ) \big\Vert_{H} + \vert v(0) \vert + \sum_{n=1}^N \vert \hat{w}_{1,n}(0) \vert + \sum_{\vert m \vert \leq M} \vert \hat{w}_{2,m}(0) \vert \Big)  \label{eq: exponential stability}
\end{align}
for every $t\geq 0$, where $H$ is either the Hilbert space $\mathcal{H}^0$ defined by \eqref{eq: state-space} or $\mathcal{H}^1$ defined by \eqref{eq: space H1}. 
\end{theorem}


\begin{remark}
As it follows from the proof, for any initial condition $(y(0,\cdot) , z(0,\cdot) , \partial_t z(0,\cdot))\in H$, there is a unique solution living in $H$.
\end{remark}


\begin{proof}
Let $\kappa\in[0,2)$ be such that 
$(c_{1,n}/n^\kappa)_{n\in\mathbb{N}^*}\in\ell^2(\mathbb{N}).$
In this proof we set $\kappa = 0$; different values will be chosen later in the proof of Theorem~\ref{thm4}. We define the observation discrepancies $e_{1,n} = w_{1,n} - \hat{w}_{1,n}$ and $e_{2,m} = w_{2,m} - \hat{w}_{2,m}$, the scaled quantity $\tilde{e}_{1,n} = n^\kappa e_{1,n}$, as well as the vectors
\begin{align*}
\hat{W}_0 & = \begin{pmatrix}
\hat{w}_{1,1} & \hat{w}_{1,2} & \ldots & \hat{w}_{1,N_0} 
\end{pmatrix}^\top , \\
E_1 & = \begin{pmatrix}
e_{1,1} & e_{1,2} & \ldots & e_{1,N_0} 
\end{pmatrix}^\top , \\
\hat{W}_2 & = \begin{pmatrix}
\hat{w}_{1,N_0 + 1} & \ldots & \hat{w}_{1,N} & \hat{w}_{2,0} & \ldots & \hat{w}_{2,-M} & \hat{w}_{2,M}
\end{pmatrix}^\top , \\
E_2 & = \begin{pmatrix}
\tilde{e}_{1,N_0 + 1} & \ldots & \tilde{e}_{1,N} & e_{2,0} & \ldots & e_{2,-M} & e_{2,M}
\end{pmatrix}^\top .
\end{align*}
We infer from \eqref{eq: spectral reduction} and \eqref{eq: controller} that
\begin{subequations}
\begin{align}
\dot{\hat{W}}_0 & = A_0 \hat{W}_0 + B_{a,0} v + B_{b,0} v_d + L C_0 E_1 + L C_1 E_2 + L \zeta_1 + L \zeta_2 , \\
\dot{E}_1 & = ( A_0 - L C_0 ) E_1 - L C_1 E_2 - L \zeta_1 - L \zeta_2 , \\
\dot{\hat{W}}_2 & = A_2 \hat{W}_2 + B_{a,1} v + B_{b,1} v_d , \\
\dot{E}_2 & = A_2 E_2 ,
\end{align}
\end{subequations} 
where 
\begin{align*}
A_2 & = \mathrm{diag}\big(
\lambda_{1,N_0 + 1} , \ldots , \lambda_{1,N} , \lambda_{2,0} , \ldots , \lambda_{2,-M} , \lambda_{2,M}
\big) , \\
B_{a,1} & = \begin{pmatrix}
a_{1,N_0+1} & \ldots & a_{1,N} & a_{2,0} & \ldots & a_{2,-M} & a_{2,M}
\end{pmatrix}^\top , \\
B_{b,1} & = \begin{pmatrix}
b_{1,N_0+1} & \ldots & b_{1,N} & b_{2,0} & \ldots & b_{2,-M} & b_{2,M}
\end{pmatrix}^\top , \\
C_1 & = \begin{pmatrix}
\tfrac{c_{1,N_0+1}}{(N_0+1)^\kappa} & \ldots & \tfrac{c_{1,N}}{N^\kappa} & c_{2,0} & \ldots & c_{2,-M} & c_{2,M}
\end{pmatrix} , \\
\zeta_1 & = \sum_{k\geq N+1} c_{1,k} w_{1,k} , \quad
\zeta_2 = \sum_{\vert l \vert \geq M+1} c_{2,l} w_{2,l} ,
\end{align*}
with $A_2 \in\mathbb{C}^{N-N_0+2M+1}$, $B_{a,1},B_{b,1} \in\mathbb{C}^{N-N_0+2M+1}$ and $C_1 \in\mathbb{C}^{N-N_0+2M+1}$. We note in particular that $\Vert C_1 \Vert = O(1)$ as $N,M \to \infty$. Defining the augmented vectors
\begin{equation*}
\hat{W}_1 = \begin{pmatrix} v \\ \hat{W}_0 \end{pmatrix}
, \quad
\tilde{L} = \begin{pmatrix} 0 \\ L \end{pmatrix} ,
\end{equation*}
we obtain
\begin{align*}
v_d & = K \hat{W}_1 , \\
\dot{\hat{W}}_1 & = (A_1+B_1 K) \hat{W}_1 + \tilde{L} C_0 E_1 + \tilde{L} C_1 E_2 + \tilde{L} \zeta_1 + \tilde{L} \zeta_2 .
\end{align*}
Finally, setting
$Y = \mathrm{col}\big( \hat{W}_1 , E_1 , \hat{W}_2 , E_2 \big) $,
we have
\begin{equation*}
\dot{Y} = F Y + \mathcal{L} \zeta_1 + \mathcal{L} \zeta_2
\end{equation*}
where 
\begin{subequations}\label{eq: def matrix F}
\begin{align*}
& F = \begin{pmatrix}
A_1 + B_1 K & \tilde{L} C_0 & 0 & \tilde{L} C_1 \\
0 & \!\! A_0 - L C_0 & 0 & - L C_1 \\
B_{a,1} \big( 1 \ 0 \ \cdots \ 0 \big) \!+\! B_{b,1} K & 0 & A_2 & 0 \\
0 & 0 & 0 & A_2
\end{pmatrix}  \\
& \mathcal{L} = \begin{pmatrix}
\tilde{L}^\top & - L^\top & 0 & 0
\end{pmatrix}^\top .
\end{align*}
\end{subequations}
Let us now define a suitable Lyapunov functional. Let $P$ be a Hermitian positive definite matrix (to be chosen later). 

For the parabolic part of the system evaluated in $L^2$ norm (i.e., for the state of the system evaluated in $\mathcal{H}^0$ norm), we define
\begin{equation}\label{eq: V L2 norm}
V(Y,w) = \bar{Y}^\top P Y + \sum_{n \geq N+1}\!  \vert w_{1,n} \vert^2 + \sum_{\vert m \vert \geq M+1}\!  \vert w_{2,m} \vert^2 .
\end{equation}
Since $\Phi$ is a Riesz basis of $\mathcal{H}^0$ (see Lemma~\ref{lem: A riesz spectral}), of dual basis $\Psi$ (see Lemma~\ref{lem: dual basis}), it follows that $\sqrt{V}$ is a norm, equivalent to the norm of $\mathbb{C}^{2N+1}\times\mathcal{H}^0$. 

For the parabolic part of the system evaluated in $H^1$ norm (i.e., for the state of the system evaluated in $\mathcal{H}^1$ norm), we define
\begin{equation}\label{eq: V H1 norm}
V(Y,w) = \bar{Y}^\top P Y + \sum_{n \geq N+1} n^2 \vert w_{1,n} \vert^2 + \sum_{\vert m \vert \geq M+1} \vert w_{2,m} \vert^2 .
\end{equation}
Since $\Phi^1$ is a Riesz basis of $\mathcal{H}^1$ (see Lemma~\ref{lemma: Riesz basis H1}), $\sqrt{V}$ is a norm, equivalent to the norm of $\mathbb{C}^{2N+1}\times\mathcal{H}^1$. Indeed, using \eqref{eq: projection system trajectory into Riesz basis}, we have
\begin{align*}
Y(t,\cdot) & = \big(w^1(t,\cdot),w^2(t,\cdot),w^3(t,\cdot)\big) \nonumber \\
& = \sum_{n\in\mathbb{N}^*} w_{1,n}(t) \phi_{1,n} + \sum_{m \in\mathbb{Z}} w_{2,m}(t) \phi_{2,m} \\
& = \sum_{n\in\mathbb{N}^*} \dfrac{n \pi}{L} w_{1,n}(t) \dfrac{L}{n \pi} \phi_{1,n} + \sum_{m \in\mathbb{Z}} w_{2,m}(t) \phi_{2,m} .
\end{align*}
The above two series converge \emph{a priori} in $\mathcal{H}^0$ norm. However, for a classical solution $W(t,\cdot)\in D(\mathcal{A})$, using Lemma~\ref{lemma: Riesz basis H1}, the series converge in $\mathcal{H}^1$ norm. Hence, thanks to Lemma~\ref{lemma: Riesz basis H1}, $\Vert W(t,\cdot) \Vert_{\mathcal{H}^1}^2$ is equivalent to $\sum_{n\in\mathbb{N}^*} n^2 \vert w_{1,n}(t) \vert^2 + \sum_{m \in\mathbb{Z}} \vert w_{2,m}(t) \vert^2$. This justifies the definition of \eqref{eq: V H1 norm}.

Since the proofs in $\mathcal{H}^0$ norm and in $\mathcal{H}^1$ norm are now similar, we focus on the second case. Setting $\tilde{Y} = \mathrm{col}(Y,\zeta_1,\zeta_2)$, the computation of the time derivative of $V$ along the system trajectories gives
\begin{align*}
& \dot{V} = \bar{\tilde{Y}}^\top 
\begin{pmatrix} 
\bar{F}^\top P + P F & P \mathcal{L} & P \mathcal{L} \\
\mathcal{L}^\top P & 0 & 0 \\
\mathcal{L}^\top P & 0 & 0
\end{pmatrix} 
\tilde{Y} \\
& \phantom{=}\, + 2 \sum_{n \geq N+1} n^2 \operatorname{Re} \big( \big( \lambda_{1,n} w_{1,n} + a_{1,n} v + b_{1,n} v_d \big) \overline{w_{1,n}} \big) \\
& \phantom{=}\, + 2 \sum_{\vert m \vert \geq M+1} \operatorname{Re} \big( \big( \lambda_{2,m} w_{2,m} + a_{2,m} v + b_{2,m} v_d \big) \overline{w_{2,m}} \big) \\
& \leq \bar{\tilde{Y}}^\top  
\begin{pmatrix} 
\bar{F}^\top P + P F & P \mathcal{L} & P \mathcal{L} \\
\mathcal{L}^\top P & 0 & 0 \\
\mathcal{L}^\top P & 0 & 0
\end{pmatrix} 
\tilde{Y} \\
& \phantom{\leq}\, + 2 \sum_{n \geq N+1} n^2 \lambda_{1,n} \vert w_{1,n} \vert^2
+ 2 \sum_{\vert m \vert \geq M+1} \underbrace{\operatorname{Re}\lambda_{2,m}}_{= \rho} \vert w_{2,m} \vert^2 \\
& \phantom{\leq}\, + \tfrac{2}{\epsilon} \bigg( \sum_{n \geq N+1} n^4 \vert w_{1,n} \vert^2 + \sum_{\vert m \vert \geq M+1} \vert w_{2,m} \vert^2 \bigg) \\
& \phantom{\leq}\, + \epsilon \bigg( \sum_{n \geq N+1} \vert a_{1,n} \vert^2 + \sum_{\vert m \vert \geq M+1} \vert a_{2,m} \vert^2 \bigg) \vert v \vert^2 \\
& \phantom{\leq}\, + \epsilon \bigg( \sum_{n \geq N+1} \vert b_{1,n} \vert^2 + \sum_{\vert m \vert \geq M+1} \vert b_{2,m} \vert^2 \bigg) \vert v_d \vert^2 
\end{align*}
with $\epsilon > 0$ arbitrary and $\rho = \tfrac{1}{2L} \log\big(\tfrac{\alpha-1}{\alpha+1}\big) = \operatorname{Re}\lambda_{2,m}$. The latter inequality has been obtained by using four times Young's inequality $ab\leq\tfrac{a^2}{2\epsilon}+\tfrac{\epsilon b^2}{2}$. Defining now $E = \begin{pmatrix} 1 & 0 & \ldots & 0 \end{pmatrix}$ and $\tilde{K} = \begin{pmatrix} K & 0 & 0 & 0 \end{pmatrix}$, we infer that 
$v = E Y$ and $v_d = \tilde{K} Y $.
Furthermore, using the Cauchy-Schwarz inequality,
\begin{align*}
\zeta_1^2 & = \bigg( \sum_{k\geq N+1} c_{1,k} w_{1,k} \bigg)^2 \leq \underbrace{\sum_{k\geq N+1} \vert c_{1,k} \vert^2}_{= S_{c,1,N}} \times \sum_{k\geq N+1} \vert w_{1,k} \vert^2 , \\
\zeta_2^2 & = \bigg( \sum_{\vert l \vert\geq M+1} c_{2,l} w_{2,l} \bigg)^2 \leq \underbrace{\sum_{\vert l \vert\geq M+1} \vert c_{2,l} \vert^2}_{= S_{c,2,M}} \times \sum_{\vert l \vert\geq M+1} \vert w_{2,l} \vert^2 .
\end{align*}
Defining $S_{a,N,M} = \sum_{n \geq N+1} \vert a_{1,n} \vert^2 + \sum_{\vert m \vert \geq M+1} \vert a_{2,m} \vert^2$ and $S_{b,N,M} = \sum_{n \geq N+1} \vert b_{1,n} \vert^2 + \sum_{\vert m \vert \geq M+1} \vert b_{2,m} \vert^2$, the combination of all above estimates gives 
\begin{equation}\label{eq: dotV}
\dot{V} + 2 \delta V \leq \bar{\tilde{Y}}^\top \Theta \tilde{Y} 
+ \sum_{n \geq N+1} n^2 \Gamma_{1,n} \vert w_{1,n} \vert^2  + \Gamma_{2,M} \sum_{\vert m \vert \geq M+1} \vert w_{2,m} \vert^2 
\end{equation}
with
\begin{subequations}
\begin{align}
\Theta & =
\begin{pmatrix} 
\Theta_{1,1} & P \mathcal{L} & P \mathcal{L} \\
\mathcal{L}^\top P & -\eta_1 & 0 \\
\mathcal{L}^\top P & 0 & -\eta_2
\end{pmatrix} , \label{eq: stab condition - Theta} \\
\Gamma_{1,n} & = 2 \Big( \lambda_{1,n} + \tfrac{n^2}{\epsilon} + \delta \Big) + \tfrac{\eta_1 S_{c,1,N}}{n^2} , \\
\Gamma_{2,M} & = 2 \Big( \rho + \tfrac{1}{\epsilon} + \delta \Big) + \eta_2 S_{c,2,M} ,  \label{eq: stab condition - Gamma2m}
\end{align}
\end{subequations}
where $\Theta_{1,1} = \bar{F}^\top P + P F + 2\delta P + \epsilon S_{a,N,M} E^\top E + \epsilon S_{b,N,M} \tilde{K}^\top \tilde{K}$ and $\eta_1,\eta_2 > 0$ arbitrary. Choosing $\epsilon > L^2/\pi^2$, we have $\Gamma_{1,n} = 2 \big( - \big( \tfrac{\pi^2}{L^2} - \tfrac{1}{\epsilon} \big) n^2 + c + \delta \big) + \tfrac{\eta_1 S_{c,1,N}}{n^2} \leq \Gamma_{1,N+1}$ for any $n \geq N+1$. Hence, $\dot{V} + 2 \delta V \leq 0$ (yielding the claimed stability estimate \eqref{eq: exponential stability}) provided there exist integers $N \geq N_0 + 1$ and $M \in\mathbb{N}$, real numbers $\epsilon > L^2/\pi^2$ and $\eta_1,\eta_2 > 0$, and a Hermitian matrix $P \succ 0$ such that
\begin{equation}\label{constraints_NM}
\Theta \preceq 0 , \quad \Gamma_{1,N+1} \leq 0 , \quad \Gamma_{2,M} \leq 0 .
\end{equation}
To conclude the proof, it remains to prove that the constraints \eqref{constraints_NM} are always feasible for $N,M$ chosen large enough. Recalling that the matrix $F$ is defined by \eqref{eq: def matrix F}, it is easy to see that $F$ is Hurwitz with eigenvalues of real part less than $-\delta < 0$. Moreover, since $\Vert C_1 \Vert = O(1)$ as $N,M \to \infty$, the application of \cite[Appendix]{lhachemi2020finite} shows that the solution $P \succ 0$ of the Lyapunov equation $\bar{F}^\top P + PF + 2\delta P = -I$ is such that $\Vert P \Vert = O(1)$ as $N,M \rightarrow + \infty$. Furthermore, $\Vert E \Vert$, $\Vert \mathcal{L} \Vert$, and $\Vert \tilde{K} \Vert$ are constants, not depending on $N,M$ while $S_{a,N,M},S_{b,N,M},S_{c,1,N},S_{c,2,M} \rightarrow 0$ as $N,M \rightarrow + \infty$. Recalling that, by assumption, $\rho < - \delta$, we take $\epsilon > L^2/\pi^2$ large enough so that $\rho + \tfrac{1}{\epsilon} + \delta < 0$. We finally take
\begin{equation*}
\eta_1 = \left\{\begin{array}{cl}
\tfrac{1}{\sqrt{S_{c,1,N}}} & \mathrm{if}\; S_{c,1,N} \neq 0, \\
N & \textrm{otherwise},
\end{array}\right.
\end{equation*}
\begin{equation*}
\eta_2 = \left\{\begin{array}{cl}
\tfrac{1}{\sqrt{S_{c,2,M}}} & \mathrm{if}\; S_{c,2,M} \neq 0, \\
M & \textrm{otherwise},
\end{array}\right.
\end{equation*}
which, in particular, implies that $\eta_1,\eta_2 \rightarrow +\infty$ while $\eta_1 S_{c,1,N} , \eta_2 S_{c,2,M} \rightarrow 0$ as $N,M \rightarrow +\infty$. With these choices, it can be seen that $\Gamma_{1,N+1} \rightarrow -\infty$ as $N \rightarrow +\infty$ and $\Gamma_{2,M} \rightarrow 2 \big( \rho + \tfrac{1}{\epsilon} + \delta \big) < 0$ as $M \rightarrow +\infty$. Hence $\Gamma_{1,N+1} \leq 0$ and $\Gamma_{2,M} \leq 0$ for all $N,M$ large enough. Finally, the Schur complement theorem applied to
\begin{equation*}
\Theta =
\begin{pmatrix} 
-I + \epsilon S_{a,N,M} \bar{E}^\top E + \epsilon S_{b,N,M} \tilde{K}^\top \tilde{K} & P \mathcal{L} & P \mathcal{L} \\
\mathcal{L}^\top P & -\eta_1 & 0 \\
\mathcal{L}^\top P & 0 & -\eta_2
\end{pmatrix} .
\end{equation*}
shows that $\Theta \preceq 0$ if and only if 
\begin{equation*}
-I + \epsilon S_{a,N,M} E^\top E + \epsilon S_{b,N,M} \tilde{K}^\top \tilde{K} 
+ \begin{pmatrix} P \mathcal{L} & P \mathcal{L} \end{pmatrix}
\begin{pmatrix} \tfrac{1}{\eta_1} & 0 \\
0 & \tfrac{1}{\eta_2} \end{pmatrix} 
\begin{pmatrix} \mathcal{L}^\top P \\ \mathcal{L}^\top P \end{pmatrix}
\preceq 0 .
\end{equation*}
We note that 
\begin{equation*}
-I + \epsilon S_{a,N,M} E^\top E + \epsilon S_{b,N,M} \tilde{K}^\top \tilde{K}
\preceq - \big( 1 - \epsilon \big( S_{a,N,M} \Vert E \Vert^2 +  S_{b,N,M} \Vert \tilde{K} \Vert^2 \big) \big) I 
\preceq - \tfrac{1}{2} I
\end{equation*}
for $N,M$ taken large enough, because $\epsilon ( S_{a,N,M} \Vert E \Vert^2 +  S_{b,N,M} \Vert \tilde{K} \Vert^2 ) \rightarrow 0$ as $N,M \rightarrow + \infty$. In this case,
\begin{align*}
& -I + \epsilon S_{a,N,M} E^\top E + \epsilon S_{b,N,M} \tilde{K}^\top \tilde{K}
+ \begin{pmatrix} P \mathcal{L} & P \mathcal{L} \end{pmatrix}
\begin{pmatrix} \tfrac{1}{\eta_1} & 0 \\
0 & \tfrac{1}{\eta_2} \end{pmatrix} 
\begin{pmatrix} \mathcal{L}^\top P \\ \mathcal{L}^\top P \end{pmatrix} \\
& \hspace{8cm} \preceq
- \tfrac{1}{2} I + \big( \tfrac{1}{\eta_1} + \tfrac{1}{\eta_2} \big) P \mathcal{L} \mathcal{L}^\top P .
\end{align*}
Recalling that $\eta_1,\eta_2 \rightarrow +\infty$ and $\Vert P \Vert = O(1)$ as $N,M \rightarrow +\infty$ while $\Vert \mathcal{L} \Vert$ is a constant, not depending on $N,M$, the latter matrix is $\preceq 0$ for $N,M$ chosen large enough. This completes the proof.
\end{proof}

\section{Extension to a pointwise measurement for the heat equation}\label{sec: extemsion}
\noindent
In this section, instead of the distributed measurement \eqref{eq: measurement heat distributed} of the reaction-diffusion PDE, we consider the case of pointwise measurement described either by
\begin{equation}\label{eq: Dirichlet measurement}
	y_o(t) = y(t,\xi_{p})
\end{equation}
or by
\begin{equation}\label{eq: Neumann measurement}
	y_o(t) = \partial_x y(t,\xi_{p})
\end{equation}
for a fixed $\xi_{p} \in [0,L]$. We infer from \eqref{eq: change of variable} and \eqref{eq: projection system trajectory into Riesz basis} that either
\begin{align*}
	y_o(t) & = w(t,\xi_{p}) \\
	& = \sum_{n\in\mathbb{N}^*} w_{1,n}(t) \underbrace{\phi_{1,n}^1(\xi_{p})}_{=c_{1,n}} + \sum_{m \in\mathbb{Z}} w_{2,m}(t) \underbrace{\phi_{2,m}^1(\xi_{p})}_{=c_{2,m}}
\end{align*}
or
\begin{align*}
	y_o(t) & = \partial_x w(t,\xi_{p}) \\
	& = \sum_{n\in\mathbb{N}^*} w_{1,n}(t) \underbrace{(\phi_{1,n}^1)'(\xi_{p})}_{=c_{1,n}} + \sum_{m \in\mathbb{Z}} w_{2,m}(t) \underbrace{(\phi_{2,m}^1)'(\xi_{p})}_{=c_{2,m}}. 
\end{align*}

\begin{theorem}\label{thm4}
Considering the pointwise measurement described either by \eqref{eq: Dirichlet measurement} or \eqref{eq: Neumann measurement} for some $\xi_p \in [0,L]$, the same statement as Theorem~\ref{thm: main result 2} holds true for the Hilbert space $H = \mathcal{H}^1$ defined by \eqref{eq: space H1}. 
\end{theorem}


\begin{proof}
The assumption that $c_{1,n} \neq 0$ for any $n\in\{1,\ldots,N_0\}$ implies the controllability of the pair $(A_0,C_0)$. Now, the proof follows the one of Theorem~\ref{thm: main result 2} by using the Lyapunov functional \eqref{eq: V H1 norm}. The only differences lie in the estimates of $\zeta_1 = \sum_{n\geq N+1} c_{1,n} w_{1,n}$ and of $\zeta_2 = \sum_{\vert m \vert \geq M+1} c_{2,m}^1(\xi_{p}) w_{2,m}$, as well as in the selection of $\kappa\in[0,2)$ that appears in the definition of the matric $C_1$ and which must be selected to ensure that $\Vert C_1 \Vert = O(1)$ as $N \rightarrow +\infty$ for applying the Lemma in the appendix of \cite{lhachemi2020finite}. 
	
In the case of the pointwise measurement \eqref{eq: Dirichlet measurement}, by Lemma~\ref{lem: eigenstructures of A}, we have $\phi_{1,n}^1(\xi_{p}) = O(1)$, hence 
\begin{equation*}
\zeta_1^2 \leq \underbrace{\sum_{n\geq N+1} \tfrac{\vert \phi_{1,n}^1(\xi_{p}) \vert^2}{n^2}}_{= S_{c,1,N} < +\infty} \times \sum_{n\geq N+1} n^2 \vert w_{1,n} \vert^2 .
\end{equation*}
Moreover, we showed in the proof of Lemma~\ref{lem: A riesz spectral} that $\Vert \phi_{2,m}^1 \Vert_{L^\infty} = O(1/m^2)$, yielding
\begin{equation*}
\zeta_2^2 \leq \underbrace{\sum_{\vert m \vert\geq M+1} \vert \phi_{2,m}^1(\xi_p) \vert^2}_{= S_{c,2,M} < +\infty} \times \sum_{\vert m \vert\geq M+1} \vert w_{2,m} \vert^2 .
\end{equation*}
Then, proceeding as in the proof of Theorem~\ref{thm: main result 2}, we obtain \eqref{eq: dotV} with 
$\Gamma_{1,n} = 2 \big( \lambda_{1,n} + \tfrac{n^2}{\epsilon} + \delta \big) + \eta_1 S_{c,1,N}$,
while $\Theta$ and $\Gamma_{2,M}$ are defined by \eqref{eq: stab condition - Theta} and \eqref{eq: stab condition - Gamma2m}, respectively. The above discussion also shows that selecting $\kappa = 1$ ensures that $\Vert C_1 \Vert = O(1)$ as $N,M \rightarrow +\infty$. The proof is then similar to the one of Theorem~\ref{thm: main result 2}.

In the case of the pointwise measurement \eqref{eq: Neumann measurement}, by Lemma~\ref{lem: eigenstructures of A}, we have $(\phi_{1,n}^1)'(\xi_{p}) = O(n)$, hence 
\begin{equation*}
\zeta_1^2 \leq \underbrace{\sum_{n\geq N+1} \tfrac{\vert (\phi_{1,n}^1)'(\xi_{p}) \vert^2}{n^{7/2}}}_{= S_{c,1,N} < \infty} \times \sum_{n\geq N+1} n^{7/2} \vert w_{1,n} \vert^2 .
\end{equation*}
Moreover, following the proof in Lemma~\ref{lem: A riesz spectral} that $\Vert (\phi_{2,m}^1)' \Vert_{L^\infty} = O(1/\vert m \vert^{3/2})$, we have
\begin{equation*}
\zeta_2^2 \leq \underbrace{\sum_{\vert m \vert\geq M+1} \vert \phi_{2,m}^1(\xi_p) \vert^2}_{= S_{c,2,M} < \infty} \times \sum_{\vert m \vert\geq M+1} \vert w_{2,m} \vert^2 .
\end{equation*}
Then, proceeding as in the proof of Theorem~\ref{thm: main result 2}, we obtain \eqref{eq: dotV} with 
$\Gamma_{1,n} = 2 \Big( \lambda_{1,n} + \tfrac{n^2}{\epsilon} + \delta \Big) + n^{3/2} \eta_1 S_{c,1,N}$,
while $\Theta$ and $\Gamma_{2,M}$ are defined by \eqref{eq: stab condition - Theta} and \eqref{eq: stab condition - Gamma2m}, respectively. Selecting $\kappa=7/4$, the above discussion also implies that $\Vert C_1 \Vert = O(1)$ as $N,M \rightarrow +\infty$. The proof is then similar to the one of Theorem~\ref{thm: main result 2}.
\end{proof}

\section{Numerical illustration}\label{sec: numerical}

We consider the wave-heat cascade system \eqref{eq: cascade equation} with $L=1$, $c = 10$, and $\beta(x)=1+x^2$. The considered output is the pointwise measurement \eqref{eq: Dirichlet measurement} localized at $\xi_p = \sqrt{3}/2$. In this setting the first element of the parabolic spectrum is $\lambda_{1,1} = 10 - \pi^2 > 0$, indicating that the heat subsystem is open-loop unstable. In order to achieve the targeted closed-loop decay rate $\delta = 1$, we set for the preliminary feedback \eqref{eq: preliminary feedback} the constant $\alpha = 1.1$, ensuring that $\mathrm{Re}\,\lambda_{2,m} < - \delta$. We next build the finite dimensional model \eqref{truncature:eq} of dimension $N_0 = 1$ that captures the eigenvalue $\lambda_{1,1}$. It is checked that $\gamma_1 \neq 0$, which enables based on Lemma~\ref{kalman:contro:lemma} the computation of a feedback matrix $K_1$ so that $A_1-B_1 K$ has its eigenvalues located at $\{-2,-3\}$. Moreover, it is checked that $c_{1,1} \neq 0$, which enables us to compute the observer gain $L$ so that $A_0 - L C_0$ has its eigenvalue located at $-4$. Writing the constraints \eqref{constraints_NM} from Theorem~\ref{thm4} into linear matrix inequalities of the decision variables $\epsilon,\eta_1,\eta_2$, it is found that the closed-loop system is exponentially stable with $\delta = 1$ for $N = 2$ and $M = 8$.

For numerical simulation, we consider the initial condition $y(0,x) = x(L-x)$ for the heat equation while $z(0,x)=\partial_t z(0,x)=0$ for the wave equation. The simulation is performed by considering the 8 first modes of the parabolic spectrum and the $2 \times 20 + 1 = 41$ first modes of the hyperbolic spectrum. The obtained results are depicted in Fig.~\ref{fig: sim}. In accordance with the theoretical predictions of Theorem~\ref{thm4}, we observe the decay of the system trajectory to zero.

\begin{figure}
\centering
\subfigure[State of the PDE $y(t,\cdot)$]{
\includegraphics[width=3.5in]{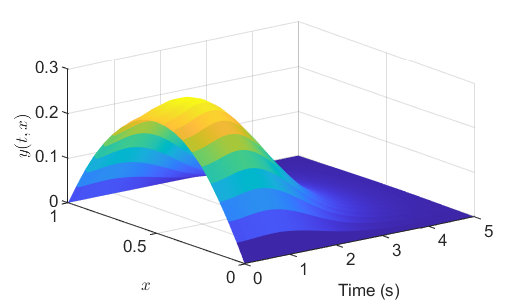}
\label{fig: sim - PDE y}
}
\subfigure[State of the PDE $z(t,\cdot)$]{
\includegraphics[width=3.5in]{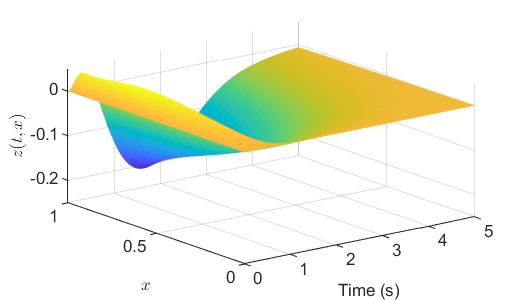}
\label{fig: sim - PDE z}
}
\subfigure[State of the observer]{
\includegraphics[width=3.5in]{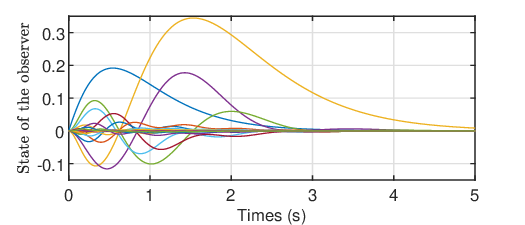}
\label{fig: sim - obs}
}
\caption{Trajectory of the wave-heat cascade \eqref{eq: cascade equation} in closed loop}
\label{fig: sim}
\end{figure}

\section{Conclusion}\label{sec: Conclusion}
\noindent
We have studied the stabilization problem for the wave-reaction-diffusion cascade system \eqref{eq: cascade equation}, where the solution of the wave equation appears as a source term in the heat equation. Exploiting the Riesz-spectral structure of the system, we derived explicit, necessary, and sufficient controllability and observability conditions under which we designed an explicit finite-dimensional output-feedback law ensuring arbitrary exponential decay in suitable Hilbert norms. 

A natural perspective is to consider the ``symmetric'' cascade model where the solution of the heat equation appears as a source term in the wave equation: 
\begin{subequations}\label{eq: cascade equation dual}
    \begin{align}
        &\partial_{tt} z(t,x) = \partial_{xx} z(t,x) + \beta(x) y(t,x) , \label{eq: cascade equation dual - 2} \\
        & \partial_t y(t,x) = \partial_{xx} y(t,x) + c y(t,x) , \label{eq: cascade equation dual - 1} \\
        & z(t,0) = \partial_x z(t,L) = 0 , \label{eq: cascade equation dual - 4} \\
        & y(t,0) = 0 , \quad y(t,L) = u(t) . \label{eq: cascade equation dual - 3}
    \end{align}
\end{subequations}
While the study of the exact controllability of the system seems to be achievable using the same approach, the actual design of an explicit feedback control strategy seems much more challenging because the preliminary shift of the spectrum of the wave equation achieved by \eqref{eq: preliminary feedback} cannot be applied anymore. This is left as an open question for future research. 

Throughout the article, we have assumed $c$ constant in \eqref{eq: cascade equation - 1}. Actually, it is not difficult to generalize all results of this paper to the case $c\in L^\infty(0,L)$, but this complicates the spectral analysis of the operators. For large values of $n$ and $m$, the eigenvalues and eigenfunctions keep asymptotically the expression given in Lemmas~\ref{lem: eigenstructures of A} and~\ref{lem: dual basis}, but for low frequencies their expression may differ, thus having an impact on the coefficients $\gamma_n$ defined by \eqref{def_gamma_n}. Then, more technical assumptions appear to characterize controllability.

Another question is to study multi-dimensional versions of \eqref{eq: cascade equation}. We believe this is much more challenging because the formalism of Riesz bases cannot be used.

\appendix

\subsection{Defining $\mathcal{B}$ by transposition}\label{sec_app_A}
At the beginning of Section~\ref{sec: prel control and spectral prop}, following the theory of well-posed linear systems developed in \cite{tucsnakweiss} (see also \cite{trelat_SB}), we have written the control system \eqref{eq: cascade equation - premilinary feedback} in the abstract form
$\dot{X}(t)=\mathcal{A} X(t)+\mathcal{B} v(t)$
with $\mathcal{A}:D(\mathcal{A})\subset\mathcal{H}^0\rightarrow\mathcal{H}^0$ defined by \eqref{def_A}, and the Hilbert space $\mathcal{H}^0$ is defined by \eqref{eq: state-space}.
In this appendix, we show how to define the control operator $\mathcal{B}$ by transposition. 
Recall that the adjoint operator $\mathcal{A}^*:D(\mathcal{A}^*)\rightarrow\mathcal{H}^0$ is defined by \eqref{def_A*}.
Following \cite{trelat_SB, tucsnakweiss}, identifying $\mathcal{H}^0$ with its dual, we search $\mathcal{B}\in L(\mathbb{R},D(\mathcal{A}^*)')$, or equivalently, $\mathcal{B}^*\in L(D(\mathcal{A}^*),\mathbb{R})$, where $D(\mathcal{A}^*)'$ is the dual of $D(\mathcal{A}^*)$ with respect to the pivot space $\mathcal{H}^0$.
Note that $D(\mathcal{A}^*) \subset \mathcal{H}^0\subset D(\mathcal{A}^*)'$ with continuous and dense embeddings.
Since the equality $\dot{X}=\mathcal{A} X+\mathcal{B} v$ is written in the space $D(\mathcal{A}^*)'$, using the duality bracket $\langle\ ,\ \rangle_{D(\mathcal{A}^*)',D(\mathcal{A}^*)}$, we have 
\begin{equation*}
\langle\dot{X}(t),\psi\rangle_{D(\mathcal{A}^*)',D(\mathcal{A}^*)} 
= \langle X(t),\mathcal{A}^*\psi\rangle_{\mathcal{H}^0} + v(t)\,\mathcal{B}^*\psi 
\end{equation*}
for any $\psi\in D(\mathcal{A}^*)$.
Taking a sufficiently regular solution $X(t,\cdot)=(y(t,\cdot),z(t,\cdot),\partial_tz(t,\cdot))^\top$ of \eqref{eq: cascade equation} (so that $\dot{X}(t)\in\mathcal{H}^0$), and denoting $\psi=(\psi_1,\psi_2,\psi_3)^\top$, a straightforward computation using \eqref{eq: cascade equation - premilinary feedback} and integrations by parts show that
$$
\langle\dot{X}(t),\psi\rangle_{\mathcal{H}^0} = \langle X(t),\mathcal{A}^*\psi\rangle_{\mathcal{H}^0} + v(t)\,\psi_3(L).
$$
Since $\langle\dot{X}(t),\psi\rangle_{D(\mathcal{A}^*)',D(\mathcal{A}^*)}  = \langle\dot{X}(t),\psi\rangle_{\mathcal{H}^0}$, a density argument finally leads to
$$
\mathcal{B}^*\psi = \psi_3(L)\qquad\forall\psi=(\psi_1,\psi_2,\psi_3)^\top\in D(\mathcal{A}^*).
$$
This formula defines the control operator $\mathcal{B}\in L(\mathbb{R},D(\mathcal{A}^*)')$ by transposition.
Following \cite[Section 5.1.4]{trelat_SB} or \cite[Proposition 10.9.1]{tucsnakweiss}, one could define $\mathcal{B}$ with an abstract formula, but this is not useful in this paper.

\subsection{Controllability properties of the wave-heat cascade \eqref{eq: cascade equation}}\label{sec: exact controllability}
The main objective of this paper is the explicit design of an output-feedback law that stabilizes the wave-heat cascade system \eqref{eq: cascade equation} with a prescribed exponential decay rate. Theorem~\ref{thm: main result 2} shows that the system \eqref{eq: cascade equation} is completely stabilizable in the energy space $\mathcal{H}^0$ defined by \eqref{eq: state-space}, i.e., can be stabilized at any decay rate by means of a linear feedback, provided that $\gamma_n\neq 0$ for all $n\in\mathbb{N}^*$. This raises the natural question of whether corresponding exact controllability and exact null-controllability properties hold (see \cite{trelat_SB,tucsnakweiss}).

We give below, for completeness, a controllability result for the cascade system \eqref{eq: cascade equation}. Its proof, which is given in the companion work \cite{LPT_ongoing}, relies on an observability inequality for the adjoint system, obtained by applying an Ingham-M\"untz type inequality \cite{ingham,komornik2015ingham,bhandari2024boundary} to Riesz-spectral expansions of the solutions.

\begin{theorem}\label{thm: observability inequality}
Assume that $T>2L$ and that
\begin{equation}\label{assumption_gamma_n}
\gamma_n\neq 0\qquad\forall n\in\mathbb{N}^* .
\end{equation}
(Recall that $\gamma_n$ is defined by \eqref{eq: def gamma_1 - 1} if $\tfrac{n^2 \pi^2}{L^2}\neq c$ and by \eqref{eq: def gamma_1 - 2} if $\tfrac{n^2 \pi^2}{L^2} = c$.)
Define the Hilbert space
\begin{equation*}
V =  \bigg\{  \sum_{n\in\mathbb{N}^*} a_n \phi_{1,n} + \sum_{m \in\mathbb{Z}} b_m \phi_{2,m}  \, \mid\,  a_n\in\mathbb{R},\; b_m\in\mathbb{C} , \; 
\sum_{n\in\mathbb{N}^*} \vert a_n \vert^2 \dfrac{n^4}{ \gamma_n^2}e^{\nu n^2} + \sum_{m\in\mathbb{Z}} \vert b_m \vert^2 < +\infty   \bigg\}
\end{equation*}
with $\nu = 2\tfrac{\pi^2}{L}\big(1+\tfrac{T}{L}\big)$, endowed with the norm 
\begin{equation*}
\Vert X \Vert_{V}^2 = \!\! \sum_{n\in\mathbb{N}^*} \! \vert \langle X , \psi_{1,n} \rangle \vert^2 \dfrac{n^4}{ \gamma_n^2}e^{\nu n^2}  \! + \! \sum_{m\in\mathbb{Z}} \vert \langle X , \psi_{2,m} \rangle \vert^2 .
\end{equation*}
The Hilbert space $V_0$ is defined as $V$ but with $\nu=\tfrac{2\pi^2}{L}$. Then:
\begin{enumerate}
\item\label{item_exact} The control system \eqref{eq: cascade equation} is exactly controllable in time $T$ in the space $V$, i.e., given any $X_0,X_1\in V$, there exists $u \in L^2(0,T)$ such that the solution $X=(y,z,\partial_t z)^\top$ of \eqref{eq: cascade equation} starting at $X(0)=X_0$ satisfies $X(T)=X_1$.
\item\label{item_exact_null} The control system \eqref{eq: cascade equation} is exactly null controllable in time $T$ in the space $V_0$, i.e., given any $X_0\in V_0$, there exists $u \in L^2(0,T)$ such that the solution $X=(y,z,\partial_t z)^\top$ of \eqref{eq: cascade equation} starting at $X(0)=X_0$ satisfies $X(T)=0$.
\item\label{item_approx} The control system \eqref{eq: cascade equation} is approximately controllable in time $T$ in the Hilbert space $\mathcal{H}^0$ (or in any other Hilbert space $H$ such that $H\subset \mathcal{H}^0$ or $\mathcal{H}^0\subset H$ with continuous and dense embeddings), i.e., given any $X_0,X_1\in \mathcal{H}^0$ and any $\varepsilon>0$, there exists $u \in L^2(0,T)$ such that the solution $X=(y,z,\partial_t z)^\top$ of \eqref{eq: cascade equation} starting at $X(0)=X_0$ satisfies $\Vert X(T)-X_1\Vert_{\mathcal{H}^0}\leq\varepsilon$ .
\end{enumerate}
Moreover, if one of the two following holds:
\begin{itemize}
\item $0<T<2L$;
\item $T > 0$ with $\gamma_n=0$ for some $n\in\mathbb{N}^*$,
\end{itemize}
then the control system \eqref{eq: cascade equation} is neither exactly controllable, nor exactly null controllable, nor approximately controllable in time $T$ in any Hilbert space $\mathcal{H}$ such that $\mathcal{H}\subset\mathcal{H}^0$ or $\mathcal{H}^0\subset\mathcal{H}$ with continuous and dense embeddings.
\end{theorem}

\begin{remark}
Thanks to Lemma~\ref{lem: A riesz spectral}, the spaces $V$ and $V_0$ coincide with $\mathcal{H}^0$ on the hyperbolic components, whereas they are strict subspaces on the parabolic component. This reflects the stronger regularity requirement imposed by the reaction-diffusion part of the cascade.

It is interesting to note that the spaces $V$ and $V_0$ are not conventional, due to the fact that, depending on the coupling function $\beta\in L^\infty(0,L)$, the coefficients $\gamma_n$ may be wildly oscillating.
\end{remark}

It is known that exact null controllability implies complete stabilizability (see \cite{liuwangxuyu,trelatwangxu}) for initial data in the same space. Hence, Theorem \ref{thm: observability inequality} implies that, under the assumption of nontriviality of the coefficients $\gamma_n$, there exists a linear feedback control $u$ such that the closed-loop system \eqref{eq: cascade equation} is exponentially stable for initial data in $V_0$, with a decay rate that can be chosen arbitrarily large.
However, this result remains purely theoretical and does not provide an explicit way to construct a feedback law. Moreover, stabilization is obtained only for initial data in $V_0$, which is a rather small subspace of $\mathcal{H}^0$. These facts further motivate the explicit output-feedback construction developed in the present article.

\paragraph{Acknowlegment.}
The third author acknowledges the support of ANR-20-CE40-0009 (TRECOS).


\bibliographystyle{IEEEtranS}
\nocite{*}
\bibliography{IEEEabrv,mybibfile}

\begin{thebibliography}{10}
\providecommand{\url}[1]{#1}
\csname url@samestyle\endcsname
\providecommand{\newblock}{\relax}
\providecommand{\bibinfo}[2]{#2}
\providecommand{\BIBentrySTDinterwordspacing}{\spaceskip=0pt\relax}
\providecommand{\BIBentryALTinterwordstretchfactor}{4}
\providecommand{\BIBentryALTinterwordspacing}{\spaceskip=\fontdimen2\font plus
\BIBentryALTinterwordstretchfactor\fontdimen3\font minus
  \fontdimen4\font\relax}
\providecommand{\BIBforeignlanguage}[2]{{%
\expandafter\ifx\csname l@#1\endcsname\relax
\typeout{** WARNING: IEEEtranS.bst: No hyphenation pattern has been}%
\typeout{** loaded for the language `#1'. Using the pattern for}%
\typeout{** the default language instead.}%
\else
\language=\csname l@#1\endcsname
\fi
#2}}
\providecommand{\BIBdecl}{\relax}
\BIBdecl

\bibitem{bhandari2021boundary}
K.~Bhandari and F.~Boyer, ``Boundary null-controllability of coupled parabolic
  systems with {R}obin conditions,'' \emph{Evolution Equations and Control
  Theory}, vol.~10, no.~1, pp. 61--102, 2021.

\bibitem{bhandari2024boundary}
K.~Bhandari, S.~Chowdhury, R.~Dutta, and J.~Kumbhakar, ``Boundary
  null-controllability of {1D} linearized compressible {N}avier-{S}tokes system
  by one control force,'' \emph{preprint hal-04625010}, 2024.

\bibitem{boyer}
F.~Boyer, ``Controllability of linear parabolic equations and systems,'' 2003,
  {M}aster, France, hal-02470625v4.

\bibitem{celuch2009modeling}
M.~Celuch and P.~Kopyt, ``Modeling microwave heating in foods,'' in
  \emph{Development of packaging and products for use in microwave
  ovens}.\hskip 1em plus 0.5em minus 0.4em\relax Elsevier, 2009, pp. 305--348.

\bibitem{chen2017backstepping}
S.~Chen, R.~Vazquez, and M.~Krstic, ``Backstepping control design for a coupled
  hyperbolic-parabolic mixed class {PDE} system,'' in \emph{IEEE 56th Annual
  Conference on Decision and Control}, 2017, pp. 664--669.

\bibitem{chowdhury2023boundary}
S.~Chowdhury, R.~Dutta, and S.~Majumdar, ``Boundary controllability and
  stabilizability of a coupled first-order hyperbolic-elliptic system.''
  \emph{Evolution Equations \& Control Theory}, vol.~12, no.~3, 2023.

\bibitem{coron2004global}
J.-M. Coron and E.~Tr{\'e}lat, ``Global steady-state controllability of
  one-dimensional semilinear heat equations,'' \emph{SIAM Journal on Control
  and Optimization}, vol.~43, no.~2, pp. 549--569, 2004.

\bibitem{coron2006global}
------, ``Global steady-state stabilization and controllability of {1D}
  semilinear wave equations,'' \emph{Communications in Contemporary
  Mathematics}, vol.~8, no.~04, pp. 535--567, 2006.

\bibitem{curtain2012introduction}
R.~F. Curtain and H.~Zwart, \emph{An introduction to infinite-dimensional
  linear systems theory}.\hskip 1em plus 0.5em minus 0.4em\relax Springer
  Science \& Business Media, 2012, vol.~21.

\bibitem{ghousein2020backstepping}
M.~Ghousein and E.~Witrant, ``Backstepping control for a class of coupled
  hyperbolic-parabolic {PDE} systems,'' in \emph{2020 American Control
  Conference (ACC)}.\hskip 1em plus 0.5em minus 0.4em\relax IEEE, 2020, pp.
  1600--1605.

\bibitem{gohberg1978introduction}
I.~C. Gohberg and M.~G. Krein, \emph{Introduction to the theory of linear
  nonselfadjoint operators}, ser. Translations of Mathematical
  Monographs.\hskip 1em plus 0.5em minus 0.4em\relax American Mathematical
  Society, Providence, RI, 1969, vol. Vol. 18, translated from the Russian by
  A. Feinstein.

\bibitem{grune2021finite}
L.~Gr{\"u}ne and T.~Meurer, ``Finite-dimensional output stabilization of linear
  diffusion-reaction systems--a small-gain approach,'' \emph{arXiv preprint
  arXiv:2104.06102}, 2021.

\bibitem{hill1996modelling}
J.~M. Hill and T.~R. Marchant, ``Modelling microwave heating,'' \emph{Applied
  Mathematical Modelling}, vol.~20, no.~1, pp. 3--15, 1996.

\bibitem{ingham}
A.~E. Ingham, ``Some trigonometrical inequalities with applications to the
  theory of series,'' \emph{Math. Z.}, vol.~41, no.~1, pp. 367--379, 1936.

\bibitem{katz2020constructive}
R.~Katz and E.~Fridman, ``Constructive method for finite-dimensional
  observer-based control of {1-D} parabolic {PDEs},'' \emph{Automatica}, vol.
  122, p. 109285, 2020.

\bibitem{komornik2015ingham}
V.~Komornik and G.~Tenenbaum, ``An {I}ngham-{M}{\"u}ntz type theorem and
  simultaneous observation problems,'' \emph{Evolution Equations and Control
  Theory}, vol.~4, no.~3, pp. 297--314, 2015.

\bibitem{lhachemi2020finite}
H.~Lhachemi and C.~Prieur, ``Finite-dimensional observer-based boundary
  stabilization of reaction-diffusion equations with either a {D}irichlet or
  {N}eumann boundary measurement,'' \emph{Automatica}, vol. 135, p. 109955,
  2022.

\bibitem{lhachemi2021nonlinear}
------, ``Nonlinear boundary output feedback stabilization of
  reaction--diffusion equations,'' \emph{Systems \& Control Letters}, vol. 166,
  p. 105301, 2022.

\bibitem{lhachemi2020pi}
H.~Lhachemi, C.~Prieur, and E.~Tr{\'e}lat, ``{PI} regulation of a
  reaction--diffusion equation with delayed boundary control,'' \emph{IEEE
  Transactions on Automatic Control}, vol.~66, no.~4, pp. 1573--1587, 2020.

\bibitem{lhachemi2025boundary}
------, ``Boundary control of heat-heat cascades,'' \emph{arXiv preprint
  arXiv:2506.10497}, 2025.

\bibitem{lhachemi2025controllability}
------, ``Controllability and stabilization of a wave-heat cascade system,''
  \emph{arXiv preprint arXiv:2506.10495}, 2025.

\bibitem{LPT_ongoing}
------, ``Controllability of a wave-heat cascade system,'' \emph{ongoing},
  2025.

\bibitem{liuwangxuyu}
\BIBentryALTinterwordspacing
H.~Liu, G.~Wang, Y.~Xu, and H.~Yu, ``Characterizations of complete
  stabilizability,'' \emph{SIAM J. Control Optim.}, vol.~60, no.~4, pp.
  2040--2069, 2022. [Online]. Available:
  \url{https://doi.org/10.1137/20M1386761}
\BIBentrySTDinterwordspacing

\bibitem{rosier2013unique}
L.~Rosier and B.-Y. Zhang, ``Unique continuation property and control for the
  {B}enjamin--{B}ona--{M}ahony equation on a periodic domain,'' \emph{Journal
  of Differential Equations}, vol. 254, no.~1, pp. 141--178, 2013.

\bibitem{russell1978controllability}
D.~L. Russell, ``Controllability and stabilizability theory for linear partial
  differential equations: recent progress and open questions,'' \emph{{SIAM}
  Review}, vol.~20, no.~4, pp. 639--739, 1978.

\bibitem{sakawa1983feedback}
Y.~Sakawa, ``Feedback stabilization of linear diffusion systems,'' \emph{SIAM
  journal on control and optimization}, vol.~21, no.~5, pp. 667--676, 1983.

\bibitem{trelat_SB}
E.~Tr\'elat, \emph{Control in finite and infinite dimension}, ser.
  SpringerBriefs on PDEs and Data Science.\hskip 1em plus 0.5em minus
  0.4em\relax Springer, Singapore, 2024.

\bibitem{trelatwangxu}
E.~Tr\'elat, G.~Wang, and Y.~Xu, ``Characterization by observability
  inequalities of controllability and stabilization properties,'' \emph{Pure
  Appl. Anal.}, vol.~2, no.~1, pp. 93--122, 2020.

\bibitem{tucsnakweiss}
M.~Tucsnak and G.~Weiss, \emph{Observation and control for operator
  semigroups}, ser. Birkh\"auser Advanced Texts: Basler Lehrb\"ucher.
  [Birkh\"auser Advanced Texts: Basel Textbooks].\hskip 1em plus 0.5em minus
  0.4em\relax Birkh\"auser Verlag, Basel, 2009.

\bibitem{wei2012optimal}
W.~Wei, H.-M. Yin, and J.~Tang, ``An optimal control problem for microwave
  heating,'' \emph{Nonlinear analysis: theory, methods \& applications},
  vol.~75, no.~4, pp. 2024--2036, 2012.

\bibitem{zhang2003polynomial}
X.~Zhang and E.~Zuazua, ``Polynomial decay and control of a 1-d model for
  fluid-structure interaction,'' \emph{C. R. Math. Acad. Sci. Paris}, vol. 336,
  no.~9, pp. 745--750, 2003.

\bibitem{zhang2004polynomial}
------, ``Polynomial decay and control of a {1-D} hyperbolic--parabolic coupled
  system,'' \emph{Journal of Differential Equations}, vol. 204, no.~2, pp.
  380--438, 2004.

\bibitem{zhong2014state}
J.~Zhong, S.~Liang, Q.~Xiong, Y.~Yuan, and C.~Zeng, ``A state space
  representation for one-dimensional microwave heating temperature model,'' in
  \emph{2014 Proceedings of the SICE Annual Conference (SICE)}.\hskip 1em plus
  0.5em minus 0.4em\relax IEEE, 2014, pp. 1366--1371.

\end{thebibliography}

\end{document}